\documentclass[12pt]{article}
\usepackage{arydshln}
\usepackage{amsmath,amsthm,amsfonts,amssymb,amscd,caption,color,subcaption,cite}
\usepackage{graphicx}
\usepackage[mathscr]{eucal}
\allowdisplaybreaks[4]
\usepackage[figurename=Fig.]{caption}
\numberwithin{equation}{section}
\allowdisplaybreaks[4]

\newtheorem {Lemma}{Lemma}[section]
\newtheorem {Theorem} {Theorem}[section]
\newtheorem {Corollary}{Corollary}[section]

\newtheorem {Claim} {Claim}[section]
\usepackage{tikz}
\usetikzlibrary{positioning}
\usepackage{circuitikz}
\usetikzlibrary{patterns}
\usepackage{fullpage}

\begin{document}

\title{The maximum index and spectral radius of unbalanced signed multipartite graphs}

\author{Yiting Cai\footnote{E-mail: yitingcai@m.scnu.edu.cn},  Bo Zhou\footnote{Corresponding author; E-mail: zhoubo@m.scnu.edu.cn}\\
School of  Mathematical Sciences, South China Normal University,\\
Guangzhou 510631, P.R. China
}

\date{}
\maketitle

\begin{abstract}
Let $\Gamma=(G,\sigma)$ be a signed graph, where $G$ is the underlying graph with vertex set $V(G)$ and edge set $E(G)$ such that $\sigma: E(G)\to \{-1,1\}$ is the sign function.  For $U\subset V(G)$,  the operation that changes the sign of all edges between $U$ and $V(G)\setminus U$ is 
called switching. Two signed graphs with the same underlying graph  are switching equivalent if one  is obtainable from  the other one by switching a subset. Two signed graphs are switching isomorphic
if one  is isomorphic to a switching equivalent signed graph of the other one. A signed cycle is called negative if it contains an odd number of negative edges. A signed graph is balanced if none of its cycles is negative; otherwise it is unbalanced. The adjacency matrix $A(\Gamma)$ of $\Gamma$ is obtained from the standard $(0,1)$-adjacency matrix of $G$ by reversing the sign of all $1$s which correspond to negative edges.
The index of $\Gamma$ is the largest eigenvalue of  $A(\Gamma)$ and the spectral radius of $\Gamma$ is  the largest absolute value of the eigenvalue of  $A(\Gamma)$. The least eigenvalue of $\Gamma$ is the least eigenvalue of $A(\Gamma)$.  We study the extremal problems of the index and the spectral radius among unbalanced signed multipartite graphs. More precisely,  we determine the unbalanced signed $t$-partite graphs with fixed $t\ge 2$ and partite sizes (order, respectively) that maximizes the index and the spectral radius respectively, up to switching isomorphism. To determine the unbalanced signed multipartite graphs with fixed partite sizes (order, respectively) with maximum spectral radius, we also determine those with minimum least eigenvalue.
\end{abstract}

\section{Introduction}

Signed graphs were introduced by Harary in 1953 in relation to certain problems in social psychology \cite{Ha}. 
Let $G$ be a simple graph with vertex set $V(G)$ and edge set $E(G)$.
A signed graph $\Gamma=(G, \sigma)$ is a graph $G$ with a function $\sigma$ defined from  $E(G)$ to $\{-1,1\}$, in which $G$ and $\sigma$ are called the underlying graph and the sign function (or signature) of $\Gamma$, respectively. 

For a signed graph $\Gamma=(G,\sigma)$, if $U\subset V(G)$, then switching  $U$ means that reversing the signs of all edges between vertices $u$ and $v$ with  $u\in U$ and $v\in V(G)\setminus U$. 
Two signed graphs  are switching equivalent if they share the same underlying graph, and one can be obtained from the other by switching  some vertex subset. 
Two signed graphs are isomorphic if there exists an isomorphism between their underlying graphs that preserves the signs of the edges. 
Two signed graphs $\Gamma_1$ and $\Gamma_2$  are switching isomorphic, denoted by $\Gamma_1 \simeq \Gamma_2$, if  $\Gamma_1$ is isomorphic to a signed graph that is switching equivalent to $\Gamma_2$. 
For a signed graph $\Gamma=(G, \sigma)$ with $e\in E(G)$, if $\sigma(e)=1$ ($\sigma(e)=-1$, respectively), then the edge $e$ is positive (negative, respectively). If each edge of a signed graph $(G,\sigma)$ is positive (negative, respectively), then it is denoted by $G^+$ ($G^-$, respectively). An ordinary graph $G$ is also viewed as $G^+$. A cycle in $\Gamma$ is positive (negative, respectively) if the number of its negative edges is even (odd, respectively).
A signed graph is called balanced (unbalanced, respectively) if it contains no (at least one, respectively) negative cycle \cite{Ha,HK,Z2}.
It follows by \cite[Lemma 5.3]{Z2} that a signed graph $(G,\sigma)$ is balanced if and
only if it is switching equivalent to $G^+$. For a signed graph $\Gamma=(G,\sigma)$, if $G$ possesses some property, we also say $\Gamma$ possesses this property. For example, $\Gamma$ is bipartite means that $G$ is bipartite.

An independent set of a graph is a subset of vertices that are pairwise nonadjacent.
Let $t$ be an integer with $t\ge2$. A graph $G$ is $t$-partite if $V(G)$ can be partitioned into $t$ independent sets,  which are called the partite sets. The cardinality of a partite set is called the partite size. 
Denote by $K_{n_1,\dots,n_t}$ the complete $t$-partite graph with partite sizes  $n_1,\dots,n_t$. Moreover, if $n_1+\dots+n_t=n$ and $|n_i-n_j|\le1$ for any $1\le i< j\le t$, then it is known as the Tur\'{a}n graph, denoted by $T_{n,t}$.

The adjacency matrix  of a signed graph $\Gamma=(G, \sigma)$ of order $n$ is defined to be the $n\times n$ matrix $A(\Gamma)=(a_{uv})_{u,v\in V(G)}$, where
\[
a_{uv}=\begin{cases} \sigma(uv) & \mbox{if $uv\in E(G)$},\\
0 & \mbox{otherwise}.
\end{cases}
\]
The eigenvalues of $\Gamma$ are just the eigenvalues of $A(\Gamma)$, which we order as  $\lambda_1(\Gamma)\ge \dots\ge\lambda_n(\Gamma)$. Many aspects of signed graphs have been studied, see, e.g. \cite{AL,StanB} (and references therein).

In this article, we study the index and spectral radius of signed graphs. Given a signed graph $\Gamma$ of order $n$, the index of $\Gamma$ is the the largest eigenvalue $\lambda_1(\Gamma)$,  and the spectral radius $\rho(\Gamma)$ of $\Gamma$ is the largest absolute value of the eigenvalues of $\Gamma$. As \[
\rho(\Gamma)=\max\{\lambda_1(\Gamma), -\lambda_n(\Gamma)\}, 
\] 
the least eigenvalue  $\lambda_n(\Gamma)$  of $\Gamma$ is also studied.

For a signed graph $\Gamma=(G,\sigma)$, let $\Gamma'$ be the signed graph obtained from $\Gamma$ by switching  $U\subset V(G)$. Then $A(\Gamma')$ is obtained  by multipliing the rows and columns 
of $A(\Gamma)$ corresponding to $U$ by $-1$, so  switching equivalent signed graphs have similar adjacency matrices, and share the same spectrum.

Let $-\Gamma$ be the signed graph obtained from $\Gamma$ by reversing the sign of each edge.
If $\Gamma$ and $-\Gamma$ are switching equivalent, then we say $\Gamma$ is sign-symmetric, and in this case,
$\rho(\Gamma)=\lambda_1(\Gamma)=-\lambda_n(\Gamma)$.  For example, 
if $\Gamma=(G,\sigma)$ is an unbalanced bipartite graph, then it is sign-symmetric, so $\rho(\Gamma)=\lambda_1(\Gamma)$. 

For an ordinary graph $H$ of order $n$, $\lambda_i(H)=\lambda_i(H^+)$, and  its index $\lambda_1(H)$ coincides with its spectral radius $\rho(H)$. 
By 
Lemma \ref{under} below, the index of a signed graph is less than or equal to the index of its underlying graph. 
Both index and spectral radius of signed graphs have been received much attention, see the recent book of Stani\'c \cite{StanB}.  
It is a natural  problem  to determine the signed graphs with maximum index or spectral radius in  classes of unbalanced signed graphs.  From a spectral perspective, there is much difference between ordinary graphs and unblalaced signed graphs.



Among the (signed) $t$-partite graphs with fixed partite sizes $n_1,\dots, n_t$, $K_{n_1, \dots, n_t}^+$ uniquely achieves the maximum index, which is a folklore result following direct from the well known Perron-Frobenius theorem, and the eigenvalues are given in \cite{De, EH}; Among the (signed) $t$-partite graphs with fixed order $n$, $T_{n,t}$ uniquely achieves the maximum index , which was drived explicitly  by Stevanovi\'c,   Gutman and Rehman \cite{SGR}, and was hinted in \cite{Ni} by showing that  if $G$ is a $K_{t+1}$-free graph of order $n$, then $\lambda(G)<\lambda(T_{n,t})$ unless $G=T_{n,t}$.

For positive integers $n_1\ge \dots \ge n_t$ and $t\ge 2$, let $\Gamma_{n_1, \dots, n_t}$  be the unbalanced signed complete  $t$-partite graph with partite sizes $n_1, \dots, n_t$  containing exactly one negative edge, say $u_0v_0$, in which $u_0$ and $v_0$ are in the partite sets of sizes $n_1$ and $n_2$, respectively. Moreover, if   $n_1+\dots+n_t=n$ and $n_1-n_t\le1$, then we write $\Gamma_n(t)$ for
 $\Gamma_{n_1, \dots, n_t}$. 
Recently, Conde, Dratman and Grippo \cite{CDG} proved that among
 unbalanced signed bipartite graphs of order $n \ge 4$, a signed graph maximizes the spectral radius if and only if it is switching isomorphic to $\Gamma_n(2)$.  In this paper, we study the extremal spectral properties of unbalanced signed multipartite graphs. The main results are listed below. 

%



\begin{Theorem}\label{N1}
Let $\Gamma$ be an unbalanced signed $t$-partite graph with partite sizes $n_1, \dots, n_t$, where $n_1\ge \dots\ge n_t$, and $t\ge 2$. Then
\[
\lambda_1(\Gamma)\le \lambda_1\left(\Gamma_{n_1, \dots, n_t}\right)
\]
with equality if and only if $\Gamma\simeq \Gamma_{n_1, \dots, n_t}$, where 
$\lambda_1\left(\Gamma_{n_1, \dots, n_t}\right)$ is the largest root of $\phi_{n_1,\dots, n_t}(x)\Pi_{i=3}^t(x+n_i)=0$ with 
\begin{align*}
\phi_{n_1,\dots, n_t}(x)&=
x^4-n_1n_2x^2+4(n_1-1)(n_2-1)\\
&\quad -\sum_{i=3}^{t}\frac{n_i}{x+n_i}\left((x^2-4)(x+n_1)(x+n_2)+4(2x+n_1+n_2-1)\right).
\end{align*}
\end{Theorem}

\begin{Theorem}\label{N2}
Let $\Gamma$ be an unbalanced signed $t$-partite graph of order $n$ with $t\ge 2$. Then
\[
\lambda_1(\Gamma)\le \lambda_1\left(\Gamma_n(t)\right)
\]
with equality if and only if $\Gamma\simeq \Gamma_n(t)$.
\end{Theorem}

%
%

%
%
%

%
%

\begin{Theorem}\label{N5}
Let $\Gamma$ be an unbalanced signed $t$-partite graph with partite sizes $n_1, \dots, n_t$, where $n_1\ge \dots\ge n_t$ and $t\ge 2$. Then
\[
\rho(\Gamma)\le 
\begin{cases}
\rho(\Gamma_{n_1,n_2}) & \mbox{if $t=2$}, \\
\rho(K_{n_1, \dots, n_t}^-) & \mbox{if $t\ge 3$}
\end{cases}
	\]
with equality if and only if 
\[
\Gamma\simeq \begin{cases}
\Gamma_{n_1,n_2} &  \mbox{when $t=2$},\\
K_{n_1, \dots, n_t}^{-} &  \mbox{when $t\ge 3$},
\end{cases}
\]
where
\[
\rho(\Gamma_{n_1,n_2})=\sqrt{\frac{n_1n_2}{2}+\frac{1}{2}\sqrt{n_1^2n_2^2-16(n_1-1)(n_2-1)}},
\]
and for $t\ge 3$, 
$\rho(K_{n_1, \dots, n_t}^-)$ is equal to the largest root of $\left(\sum_{i=1}^t\frac{n_i}{x+n_i}-1\right)\prod_{i=1}^t(x+n_i)=0$.
\end{Theorem}

\begin{Theorem}\label{N6}
	Let $\Gamma$ be an unbalanced signed $t$-partite graph of order $n$ with $t\ge 2$. Then
	\[
	\rho(\Gamma)\le \begin{cases} \rho(\Gamma_n(2)) & \mbox{if $t=2$},\\
\rho (T_{n, t}^-) & \mbox{if $t\ge 3$}
\end{cases}
	\]
	with equality if and only if $\Gamma\simeq \begin{cases} 
\Gamma_n(2) & \mbox{when $t=2$},\\ 
T_{n, t}^- &  \mbox{when $t\ge 3$}.
\end{cases}$
\end{Theorem}

If $t=2$ in Theorem \ref{N2} (or Theorem \ref{N6}), we recover the result in \cite{CDG} as  $\lambda_1(\Gamma)=\rho(\Gamma)$ when $\Gamma$ is a unbalanced signed bipartite graph. 

Note that  $A(K_{n_1, \dots, n_t}^-)=-A(K_{n_1, \dots, n_t}^+)$. So with $n=n_1+\dots+n_t$, we have 
\[
\rho (K_{n_1, \dots, n_t}^-) =\max\{\lambda_1(K_{n_1, \dots, n_t}^-), -\lambda_n(K_{n_1, \dots, n_t}^-)\}=\lambda_1(K_{n_1, \dots, n_t}^+)
=\rho(K_{n_1, \dots, n_t}^+).
\] 
Recall that  $K_{n_1, \dots, n_t}^+$ uniquely achieves the maximum spectral radius among the balanced $t$-partite graphs with fixed partite sizes $n_1,\dots, n_t$.  From Lemma 2.3,  we have the following corollary of Theorems \ref{N5} and \ref{N6}.

\begin{Corollary}\label{N51}
(i) If  $\Gamma$ is a signed $t$-partite graph with partite sizes $n_1, \dots, n_t$, where $n_1\ge \dots\ge n_t$ and $t\ge 2$, then
\[
\rho(\Gamma)\le \rho (K_{n_1, \dots, n_t}^+)
\]
with equality if and only if $\Gamma\simeq K_{n_1, \dots, n_t}^+$, or $t\ge 3$ and $\Gamma\simeq K_{n_1, \dots, n_t}^-$.

(ii) If $\Gamma$ is a signed $t$-partite graph of order $n$ with $t\ge 2$, then
\[
\rho(\Gamma)\le \rho(T_{n,t}^+)
\]
with equality if and only if $\Gamma\simeq T_{n,t}^+$, or $t\ge 3$ and $\Gamma\simeq
T_{n, t}^-$.
\end{Corollary}

To prove Theorems \ref{N5} and \ref{N6}, we also prove the following results on the least eigenvalue, which are also of independent  interest. 
For an unbalanced signed $2$-partite graph $\Gamma$ of order $n$,  $\lambda_n(\Gamma)=-\lambda_1(\Gamma)$ is minimum if and only if $\lambda_1(\Gamma)$ is maximum, see Theorems \ref{N1} and \ref{N2}. So it suffices to consider unbalanced signed $t$-partite graphs with $t\ge 3$.

\begin{Theorem}\label{N3}
Let $\Gamma$ be an unbalanced signed $t$-partite graph with partite sizes $n_1, \dots, n_t$, where $n_1\ge \dots\ge n_t$ and $t\ge 3$. Then for $n=n_1+\dots+n_t$, 
	\[
	\lambda_n(\Gamma)\ge \lambda_n (K_{n_1, \dots, n_t}^-)
	\]
	with equality if and only if $\Gamma\simeq K_{n_1, \dots, n_t}^-$.
\end{Theorem}

\begin{Theorem}\label{N4}
Let $\Gamma$ be an unbalanced signed $t$-partite graph of order $n$ with $t\ge 3$. Then
\[
\lambda_n(\Gamma)\ge \lambda_n (T_{n, t}^-)
\]
with equality if and only if $\Gamma\simeq T_{n, t}^-$.
\end{Theorem}

The paper is organized as follows. In Section 2, we present some preliminary results. In Section 3, we establish a key  lemma on the existence of a non-negative eigenvector associated with the index of a particular unbalanced signed complete multipartite graph. In Section 4, we prove Theorems \ref{N1} and \ref{N2}. In Section 5, we prove Theorems \ref{N3}, \ref{N4}, \ref{N5} and \ref{N6}.

\section{Preliminaries}

For a real matrix $M_{n\times n}$ with $n$ real eigenvalues, denote by  $\mu_i(M)$ the $i$-th largest eigenvalue of $M$.

\begin{Lemma}\label{interlace}\cite{HJ}
Let $M$ be a real symmetric matrix of order $n$ and $M'$ a principal submatrix of $M$. Then 
\[
\mu_1(M)\ge \mu_1(M')
\]
\end{Lemma}

Let $B$ be an $n\times n$ matrix whose rows and columns are indexed by elements in $X=\{1, \dots, n\}$.  Let $\pi=X_1\cup \dots\cup X_m$ be a partition of $X$. For $1\le i,j\le m$, let $B_{ij}$ be the submatrix of $B$ whose rows and columns are indexed by elements of $X_i$ and $X_j$. The partition $\pi$ is equitable if the row sum of each $B_{ij}$ is a constant for $1\le i,j\le m$. The $m\times m$ matrix whose $(i,j)$-entry is the average row sum of $B_{ij}$ with $1\le i,j\le m$ is called the quotient matrix of $B$.

\begin{Lemma}\label{QM}\cite[Lemma 2.3.1]{BH}
Let $B$ be a real symmetric matrix. Then the spectrum of the quotient matrix of $B$ with respect to an equitable partition is contained in the spectrum of $B$.
\end{Lemma}

\begin{Lemma}\label{under}\cite{St1}
For a signed graph $\Gamma=(G,\sigma)$, we have $\lambda_1(\Gamma)\le\lambda_1(G)$ with equality when $G$ is connected if and only if $\Gamma$ is balanced.
\end{Lemma}

\begin{Lemma}\label{L+}\cite{Stan,SL}
Given a signed graph $\Gamma$ and any of its eigenvalue $\lambda$, $\Gamma$  is switching equivalent to some $\Gamma^*$ such that as an eigenvalue of $\Gamma^*$, $\lambda$ has a non-negative eigenvector.
\end{Lemma}

\begin{Lemma}\label{SE}\cite{Z1}
Two signed graphs  are switching equivalent if and only if each cycle has the same sign in both signed graphs.
\end{Lemma}


Given a signed graph $\Gamma=(G, \sigma)$ and $\emptyset\ne V\subseteq V(G)$, let $G[V]$ be the subgraph of $G$ induced by $V$,  the signed graph $\Gamma[V]=(G[V], \sigma_V)$ is called the signed subgraph of $\Gamma$  induced by $V$, where $\sigma_V$ is the restriction of $\sigma$ on $E(G[V])$.

For integer  $t\ge 2$, let $\dot\Gamma_{n_1, \dots, n_t}$ be the  unbalanced signed complete  $t$-partite graph with partite sets  $V_1, \dots, V_t$ such that $n_i=|V_i|$ for $i=1, \dots, t$ and  $u_1u_2$ is the only negative edge with  $u_1\in V_1$, $u_2\in V_2$ and $n_1\ge n_2$. 

 For $t=2$, as $\dot\Gamma_{n_1,  n_2}$ is  unbalanced, one has $n_1\ge n_2\ge 2$.

\begin{Lemma}\label{gnt} For positive integers $n_1,\dots, n_t$ with $t\ge 2$, 
$\lambda_1(\dot\Gamma_{n_1, \dots, n_t})$ is the largest root of $\phi_{n_1,\dots, n_t}(x)\Pi_{i=3}^t(x+n_i)=0$, where 
\begin{align*}\phi_{n_1,\dots, n_t}(x)&=x^4-n_1n_2x^2+4(n_1-1)(n_2-1) \\
& \quad -\sum_{i=3}^{t}\frac{n_i}{x+n_i}\left((x^2-4)(x+n_1)(x+n_2)+4(2x+n_1+n_2-1)\right). 
\end{align*} 
Moreover, if $t\ge 3$ $\lambda_1(\dot\Gamma_{n_1, \dots, n_t})\ge \tau_{n_1,n_2}$ with equality if $t=3$ and $n_3=1$, 
where $\tau_{n_1,n_2}$ is the largest root of $h_{n_1,n_2}(x)=0$ with 
\[
h_{n_1,n_2}(x)=x^4-2x^3-(n_1+n_2+n_1n_2-4)x^2+2(n_1+n_2-2)x-4n_1-4n_2+4n_1n_2+4, 
\]
and $\lambda_1(\dot\Gamma_{n_1, \dots, n_t})\ge n_2$. 
\end{Lemma}

\begin{proof} Let \[
 A=\begin{pmatrix}
1 & 1 & n_1-1 & 1 & n_2-1 \\
1 & x+1 & n_1-1 & 2 & 0\\
1 & 1 & x+n_1-1 & 0 & 0 \\
1 & 2 & 0 & x+1 & n_2-1 \\
1 & 0 & 0 & 1 & x+n_2-1
\end{pmatrix}, 
\]
and for $t\ge 3$, 
\[
B=\begin{pmatrix}
 n_3 & \cdots & n_t\\
 0 & \cdots & 0\\
 0 & \cdots & 0\\
 0 & \cdots & 0\\
 0 & \cdots & 0\\
\end{pmatrix},\\
 C=\begin{pmatrix}
1 & 0 & 0 & 0 & 0 \\
\vdots & \vdots & \vdots & & \vdots\\
1 & 0 & 0 & 0 & 0 
\end{pmatrix}, 
D=\begin{pmatrix}
x+n_3 & \cdots & 0\\
 \vdots & & \vdots\\
0 & \cdots & x+n_t
\end{pmatrix},
\]
where $C$ is a $(t-2)\times 5$ matrix. For $t\ge 3$,
\begin{align*}
&\quad |A-BD^{-1}C|\\
&=\left| \begin{array}{ccccc}
1-\sum_{i=3}^{t}\frac{n_i}{x+n_i} & 1 & n_1-1 & 1 & n_2-1\\
1 & x+1 & n_1-1 & 2 & 0\\
1 & 1 & x+n_1-1 & 0 & 0\\
1 & 2 & 0 & x+1 & n_2-1\\
1 & 0 & 0 & 1 & x+n_2-1
\end{array} \right| \\
& =\left(1-\sum_{i=3}^{t}\frac{n_i}{x+n_i}\right) \left|\begin{array}{cccc}
x+1 & n_1-1 & 2 & 0\\
1 & x+n_1-1 & 0 & 0\\
2 & 0 & x+1 & n_2-1\\
0 & 0 & 1 & x+n_2-1
\end{array} \right|\\
& \quad 
-\left| \begin{array}{cccc}
1 & n_1-1 & 2 & 0\\
1 & x+n_1-1 & 0 & 0\\
1 & 0 & x+1 & n_2-1\\
1 & 0 & 1 & x+n_2-1
\end{array} \right|
+(n_1-1) \left| \begin{array}{cccc}
1 & x+1 & 2 & 0\\
1 & 1 & 0 & 0\\
1 & 2 & x+1 & n_2-1\\
1 & 0 & 1 & x+n_2-1
\end{array} \right| \\
& \quad -\left| \begin{array}{cccc}
1 & x+1 & n_1-1 & 0\\
1 & 1 & x+n_1-1 & 0\\
1 & 2 & 0 & n_2-1\\
1 & 0 & 0 & x+n_2-1
\end{array} \right|
+(n_2-1)\left| \begin{array}{ccccc}
1 & x+1 & n_1-1 & 2 \\
1 & 1 & x+n_1-1 & 0 \\
1 & 2 & 0 & x+1 \\
1 & 0 & 0 & 1 
\end{array} \right|\\
& = \phi_{n_1,\dots, n_t}(x). 
\end{align*}
For $t=2$, $|A-BD^{-1}C|=|A|=\phi_{n_1,n_2}(x)$.

Note that $n_1-1$, $n_2-1$, $n_3, \dots, n_t$ rows of $A(\dot\Gamma_{n_1, \dots,  n_t})$ corresponding to the vertices in $V_1\setminus\{u_1\}$, $V_2\setminus\{u_2\}$, $V_3, \dots, V_t$ respectively are equal. So the rank of $A(\dot\Gamma_{n_1, \dots,  n_t})$ is at most $t+2$, that is, $0$ is the eigenvalue of $\dot\Gamma_{n_1, \dots,  n_t}$ with multiplicity at least $n-t-2$. 
We partition $V(\dot\Gamma_{n_1, \dots,  n_t})$ as $\{u_1\}\cup (V_1\setminus\{u_1\})\cup \{u_2\}\cup (V_2\setminus\{u_2\})\cup V_3\cup \dots \cup V_t$. With respect to this partition, the quotient matrix of $A(\dot\Gamma_{n_1, \dots,  n_t})$ is 
\[
Q=\left(\begin{matrix}
0 & 0 & -1 & n_2-1 & n_3 & \cdots & n_t\\
0 & 0 & 1 & n_2-1 & n_3 & \cdots & n_t\\
-1 & n_1-1 & 0 & 0 & n_3 & \cdots & n_t\\
1 & n_1-1 & 0 & 0 & n_3 & \cdots & n_t\\
1 & n_1-1 & 1 & n_2-1 & 0 & \cdots & n_t\\
\vdots & \vdots & \vdots & \vdots & \vdots & & \vdots\\
1 & n_1-1 & 1 & n_2-1 & n_3 & \cdots & 0\\
\end{matrix}\right)
\]
with characteristic polynomial 
\begin{align*}
\Phi & =
\left| \begin{array}{cccccccc}
x & 0 & 1 & -(n_2-1) & -n_3 & \cdots & -n_t\\
0 & x & -1 & -(n_2-1) & -n_3 & \cdots & -n_t\\
1 & -(n_1-1) & x & 0 & -n_3 & \cdots & -n_t\\
-1 & -(n_1-1) & 0 & x & -n_3 & \cdots & -n_t\\
-1 & -(n_1-1) & -1 & -(n_2-1) & x & \cdots & -n_t\\
\vdots & \vdots & \vdots & \vdots & \vdots & & \vdots\\
-1 & -(n_1-1) & -1 & -(n_2-1) & -n_3 & \cdots & x\\
\end{array} \right|\\
&=
\left| \begin{array}{c:cccccccc}
1 & 1 & n_1-1 & 1 & n_2-1 & n_3 & \cdots & n_t\\
 \hdashline
0 & x & 0 & 1 & -(n_2-1) & -n_3 & \cdots & -n_t\\
0 & 0 & x & -1 & -(n_2-1) & -n_3 & \cdots & -n_t\\
0 & 1 & -(n_1-1) & x & 0 & -n_3 & \cdots & -n_t\\
0 & -1 & -(n_1-1) & 0 & x & -n_3 & \cdots & -n_t\\
0 & -1 & -(n_1-1) & -1 & -(n_2-1) & x & \cdots & -n_t\\
\vdots & \vdots & \vdots & \vdots & \vdots & \vdots & & \vdots\\
0 & -1 & -(n_1-1) & -1 & -(n_2-1) & -n_3 & \cdots & x\\
\end{array} \right| \\
& = \left| \begin{array}{ccccc:cccc}
1 & 1 & n_1-1 & 1 & n_2-1 & n_3 & \cdots & n_t\\
1 & x+1 & n_1-1 & 2 & 0 & 0 & \cdots & 0\\
1 & 1 & x+n_1-1 & 0 & 0 & 0 & \cdots & 0\\
1 & 2 & 0 & x+1 & n_2-1 & 0 & \cdots & 0\\
1 & 0 & 0 & 1 & x+n_2-1 & 0 & \cdots & 0\\
 \hdashline
1 & 0 & 0 & 0 & 0 & x+n_3 & \cdots & 0\\
\vdots & \vdots & \vdots & \vdots & \vdots & \vdots & & \vdots\\
1 & 0 & 0 & 0 & 0 & 0 & \cdots & x+n_t\\
\end{array} \right| \\
& =\phi_{n_1,\dots, n_t}(x)\Pi_{i=3}^t(x+n_i). 
\end{align*}
It is easy to see that the above partition is equitable, so by Lemma \ref{QM}, the $\lambda_1(\dot\Gamma_{n_1, \dots,  n_t})$ is equal to the  largest root of $\phi_{n_1,\dots, n_t}(x)\Pi_{i=3}^t(x+n_i)=0$. 

If $t=3$ and $n_3=1$, then 
\[
\phi_{n_1,n_2,1}(x)=\frac{x+2}{x+1}h_{n_1,n_2}(x),  
\]
and so $\lambda_1(\dot\Gamma_{n_1, n_2, 1})=\tau_{n_1,n_2}$. 
Let $\Gamma_1=\dot\Gamma_{n_1, \dots, n_t}[V_1\cup V_2\cup \{v_3\}]$, where $v_3\in V_3$. 
So for $t\ge 3$, we have by Lemma \ref{interlace} that 
\[
\lambda_1(\dot\Gamma_{n_1, \dots, n_t})\ge \lambda_1(\Gamma_1)=\tau_{n_1,n_2}.
\]
As $n_1\ge n_2\ge 1$, and 
$\phi_{n_1,n_2,1}(n_2)=-\frac{(n_2+2)(n_2-1)}{n_2+1}\left((n_1-n_2)n_2^2+2(n_1+n_2)(n_2-2)+4\right)\le 0$, we have  $\tau_{n_1,n_2}\ge n_2$.
So $\lambda_1(\dot\Gamma_{n_1, \dots, n_t})\ge n_2$ with $t\ge 3$. 
%
\end{proof}

\section{A lemma on the existence of a non-negative eigenvector}

%
%
The following lemma plays a key role in our proof.

\begin{Lemma}\label{nne}  For positive integers $n_1,\dots, n_t$ with 
 $t\ge 2$, $n_1+\dots+n_t=n\ge 4$ and $(n,t)\ne (4,2)$, and a signed graph $\Gamma=(K_{n_1,\dots, n_t}, \sigma)$ containing  a unique negative edge connecting  $u\in V_i$ and $v\in V_j$ with $|V_i|\ge |V_j|$,
there is a non-negative eigenvector associated with $\lambda_1(\Gamma)$. 
Moreover, if $\mathbf{x}$ is a non-negative eigenvector associated with $\lambda_1(\Gamma)$, then it is positive unless
%
%
$t=2$ and $n_j=2$, or
$t=3$, $n_i\ge 2$, $n_j=n_\ell=1$ with $\ell\in \{1,2,3\}\setminus\{i,j\}$, and $u$ is the unique vertex with $0$ entry.
\end{Lemma}

\begin{proof} Assume that $n_1\ge n_2$, $i=1$,  $j=2$,  $u=u_1$  and $v=v_1$. Then 
$\Gamma=\dot\Gamma_{n_1, \dots, n_t}$. Recall that $u_1u_2$ is the only negative edge of $\dot\Gamma_{n_1, \dots, n_t}$, where  $u_1\in V_1$, $u_2\in V_2$. 

Suppose first that $n=4, 5$ and $3\le t\le n$. The nonisomorphic ones in $\{\dot\Gamma_{n_1, \dots, n_t}: 4\le n\le 5\}$ are shown in Fig.~\ref{f1}, where solid edges represent positive edges while dashed edges represent negative edges. For each such signed graph, the index is listed below, and an eigenvector affording the index is illustrated via the valuations beside the vertices.
Each one in  Fig.~\ref{f1} has a non-negative eigenvector associated with the index, which are positive except for $\dot\Gamma_{n_1, 1, 1}$ with $n_1=2,3$ that contain exactly one zero entry at $u_1$. So the results follow for $n=4, 5$ and $3\le t\le n$.

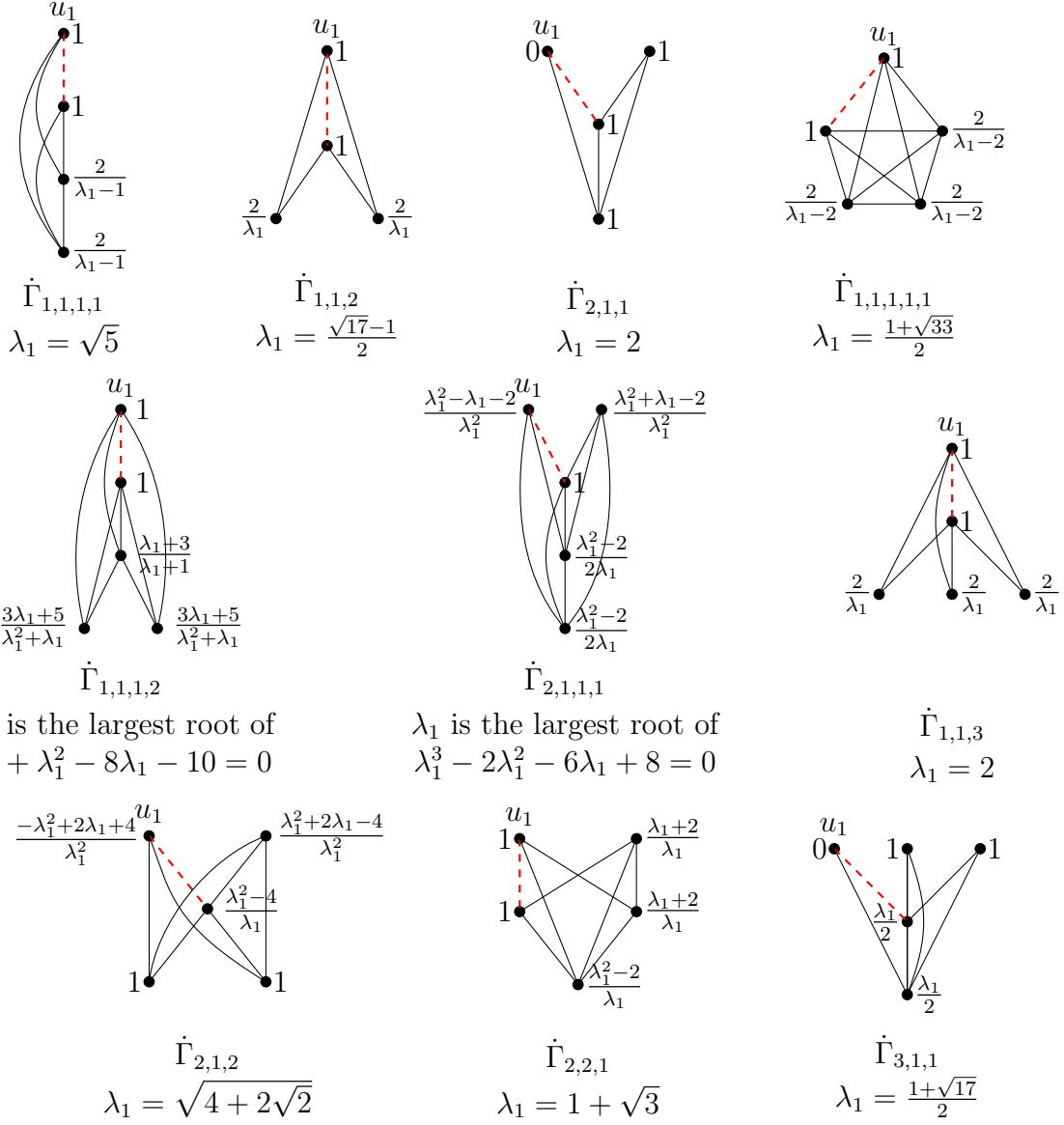
\begin{figure}[htbp]
\centering
\begin{tikzpicture}
\filldraw [black] (0,2) circle (2pt);
\filldraw [black] (0,1) circle (2pt);
\filldraw [black] (0,0) circle (2pt);
\filldraw [black] (0,-1) circle (2pt);
\draw  [red,line width=0.8pt][dashed] (0,2)--(0,1);
\node at (0,2.3) {$u_1$};
\draw  [black](0,1)--(0,0)--(0,-1);
\draw [black](0,2) .. controls (-0.5,1.2) and (-0.5,0.8).. (0,0);
\draw [black](0,2) .. controls (-0.8,1) and (-0.8,0).. (0,-1);
\draw [black](0,1) .. controls (-0.5,0.2) and (-0.5,-0.2).. (0,-1);
\node at (0, -1.6) {$\dot\Gamma_{1,1,1,1}$};
\node at (0, -2.2) {$\lambda_1=\sqrt{5}$};
\node at (0.2, 2) {$1$};
\node at (0.2, 1) {$1$};
\node at (0.5, 0) {$\frac{2}{\lambda_1-1}$};
\node at (0.5, -1) {$\frac{2}{\lambda_1-1}$};
\end{tikzpicture} \ \ \ \ \ \ \ \
\begin{tikzpicture}
\filldraw [black] (0,2.3) circle (2pt);
\filldraw [black] (0,1) circle (2pt);
\filldraw [black] (-0.7,0) circle (2pt);
\filldraw [black] (0.7,0) circle (2pt);
\draw  [red,line width=0.8pt][dashed] (0,2.3)--(0,1);
\node at (0,2.6) {$u_1$};
\draw  [black](0,1)--(-0.7,0)--(0,2.3)--(0.7,0);
\draw  [black](0,1)--(0.7,0);
\node at (0, -1) {$\dot\Gamma_{1,1,2}$};
\node at (0, -1.6) {$\lambda_1=\frac{\sqrt{17}-1}{2}$};
\node at (0.2, 2.3) {$1$};
\node at (0.2, 1) {$1$};
\node at (1, 0) {$\frac{2}{\lambda_1}$};
\node at (-1, 0) {$\frac{2}{\lambda_1}$};
\end{tikzpicture} \ \ \ \ \ \ \ \
\begin{tikzpicture}
\filldraw [black] (-0.2,2.3) circle (2pt);
\filldraw [black] (0.5,1.3) circle (2pt);
\filldraw [black] (0.5,0) circle (2pt);
\filldraw [black] (1.2,2.3) circle (2pt);
\draw  [red,line width=0.8pt][dashed] (-0.2,2.3)--(0.5,1.3);
\node at (-0.2,2.6) {$u_1$};
\draw  [black](0.5,1.3)--(0.5,0)--(1.2,2.3)--(0.5,1.3);
\draw  [black](-0.2,2.3)--(0.5,0);
\node at (0.5, -1.1) {$\dot\Gamma_{2,1,1}$};
\node at (0.5, -1.7) {$\lambda_1=2$};
\node at (-0.4, 2.3) {$0$};
\node at (1.4, 2.3) {$1$};
\node at (0.7, 1.3) {$1$};
\node at (0.7, 0) {$1$};
\end{tikzpicture} \ \ \ \ \ \ \ \
\begin{tikzpicture}
\filldraw [black] (0,2) circle (2pt);
\filldraw [black] (-0.8,1) circle (2pt);
\filldraw [black] (0.8,1) circle (2pt);
\filldraw [black] (-0.5,0) circle (2pt);
\filldraw [black] (0.5,0) circle (2pt);
\draw  [red,line width=0.8pt][dashed] (0,2)--(-0.8,1);
\node at (0,2.3) {$u_1$};
\draw  [black](-0.8,1)--(-0.5,0)--(0.5,0)--(0.8, 1)--(0,2)--(-0.5,0)--(0.8, 1)--(-0.8,1)--(0.5,0)--(0,2);
\node at (0, -1.2) {$\dot\Gamma_{1,1,1,1,1}$};
\node at (0, -1.8) {$\lambda_1=\frac{1+\sqrt{33}}{2}$};
\node at (0.2, 2) {$1$};
\node at (-1, 1) {$1$};
\node at (1.3, 1) {$\frac{2}{\lambda_1-2}$};
\node at (-1, 0) {$\frac{2}{\lambda_1-2}$};
\node at (1, 0) {$\frac{2}{\lambda_1-2}$};
\end{tikzpicture} \ \ \ \ \ \ \ \
\begin{tikzpicture}
	\filldraw [black] (0.5,2) circle (2pt);
	\filldraw [black] (0.5,1) circle (2pt);
	\filldraw [black] (0.5,0) circle (2pt);
	\filldraw [black] (1,-1) circle (2pt);
	\filldraw [black] (0,-1) circle (2pt);
	\draw  [red,line width=0.8pt][dashed] (0.5,2)--(0.5,1);
	\node at (0.5,2.3) {$u_1$};
	\draw [black](0.5,2) .. controls (0.2,1.2) and (0.2,0.8).. (0.5,0);
	\draw  [black](0.5,1)--(0.5,0)--(0,-1)--(0.5,1);
	\draw  [black](0.5,0)--(1,-1)--(0.5,1);
	\draw [black](0.5,2) .. controls (1.2,1) and (1.2,-0.2).. (1,-1);
	\draw [black](0.5,2) .. controls (-0.2,1) and (-0.2,-0.2).. (0,-1);
	\node at (0.5, -1.7) {$\dot\Gamma_{1,1,1,2}$};
	\node at (0.5, -2.6) {
		$\begin{array}{c}
			\mbox{$\lambda_1$ is the largest root of} \\
			\lambda_1^3 + \lambda_1^2 - 8\lambda_1 - 10 = 0
		\end{array}$
	};
	\node at (0.8, 2) {$1$};
	\node at (0.8, 1) {$1$};
	\node at (1.1, 0) {$\frac{\lambda_1+3}{\lambda_1+1}$};
	\node at (-0.7, -1) {$\frac{3\lambda_1+5}{\lambda_1^2+\lambda_1}$};
	\node at (1.7, -1) {$\frac{3\lambda_1+5}{\lambda_1^2+\lambda_1}$};
\end{tikzpicture} \ \ \ \ \ \ \ \
\begin{tikzpicture}
	\filldraw [black] (0,2) circle (2pt);
	\filldraw [black] (0.5,1) circle (2pt);
	\filldraw [black] (0.5,0) circle (2pt);
	\filldraw [black] (1,2) circle (2pt);
	\filldraw [black] (0.5,-1) circle (2pt);
	\draw  [red,line width=0.8pt][dashed] (0,2)--(0.5,1);
	\node at (0,2.3) {$u_1$};
	\draw  [black](0.5,1)--(0.5,0)--(1,2)--(0.5,1);
	\draw  [black](0,2)--(0.5,0);
	\draw  [black](0.5,-1)--(0.5,0);
	\draw [black](0.5,1) .. controls (0.15,0.2) and (0.15,-0.2).. (0.5,-1);
	\draw [black](1,2) .. controls (1.2,1) and (1.2,-0.2).. (0.5,-1);
	\draw [black](0,2) .. controls (-0.2,1) and (-0.2,-0.2).. (0.5,-1);
	\node at (0.5, -1.7) {$\dot\Gamma_{2,1,1,1}$};
	\node at (0.5, -2.6) {
		$\begin{array}{c}
			\mbox{$\lambda_1$ is the largest root of}\\ \lambda_1^3-2\lambda_1^2-6\lambda_1+8=0
		\end{array}$
	};
	\node at (-0.8, 1.95){$\frac{\lambda_1^2-\lambda_1-2}{\lambda_1^2}$};
	\node at (1.8, 1.95){$\frac{\lambda_1^2+\lambda_1-2}{\lambda_1^2}$};
	\node at (0.7, 1){$1$};
	\node at (1, 0){$\frac{\lambda_1^2-2}{2\lambda_1}$};
	\node at (1, -1){$\frac{\lambda_1^2-2}{2\lambda_1}$};
\end{tikzpicture} \ \ \ \ \ \ \ \
\begin{tikzpicture}
	\filldraw [black] (0,2) circle (2pt);
	\filldraw [black] (0,1) circle (2pt);
	\filldraw [black] (0,0) circle (2pt);
	\filldraw [black] (-1,0) circle (2pt);
	\filldraw [black] (1,0) circle (2pt);
	\draw  [red,line width=0.8pt][dashed] (0,2)--(0,1);
	\node at (0,2.3) {$u_1$};
	\draw  [black](0,1)--(0,0);
	\draw  [black](0,1)--(-1,0)--(0,2)--(1,0)--(0,1);
	\draw [black](0,2) .. controls (-0.3,1.2) and (-0.3,0.8).. (0,0);
	\node at (0, -1.8) {$\dot\Gamma_{1,1,3}$};
	\node at (0, -2.4) {$\lambda_1=2$};
	\node at (0.2, 2) {$1$};
	\node at (0.2, 1) {$1$};
	\node at (-1.3, 0) {$\frac{2}{\lambda_1}$};
	\node at (0.3, 0) {$\frac{2}{\lambda_1}$};
	\node at (1.3, 0) {$\frac{2}{\lambda_1}$};
\end{tikzpicture}
\begin{tikzpicture}
	\filldraw [black] (0,2) circle (2pt);
	\filldraw [black] (0,0) circle (2pt);
	\filldraw [black] (0.8,1) circle (2pt);
	\filldraw [black] (1.6,2) circle (2pt);
	\filldraw [black] (1.6,0) circle (2pt);
	\draw  [red,line width=0.8pt][dashed] (0,2)--(0.8,1);
	\node at (0,2.3) {$u_1$};
	\draw  [black](0,0)--(0.8,1)--(1.6,0)--(1.6,2)--(0.8,1);
	\draw  [black](0,0)--(0,2);
	\draw [black](0,2) .. controls (0.3,1.2) and (0.3,0.8).. (1.6,0);
	\draw [black](1.6,2) .. controls (0.8,1.7) and (0.1,0.8).. (0,0);
	\node at (0.8, -1) {$\dot\Gamma_{2,1,2}$};
	\node at (0.8, -1.6) {$\lambda_1=\sqrt{4+2\sqrt{2}}$};
	\node at (-1, 1.95) {$\frac{-\lambda_1^2+2\lambda_1+4}{\lambda_1^2}$};
	\node at (2.5, 2) {$\frac{\lambda_1^2+2\lambda_1-4}{\lambda_1^2}$};
	\node at (1.4, 1) {$\frac{\lambda_1^2-4}{\lambda_1}$};
	\node at (-0.2, 0) {$1$};
	\node at (1.8, 0) {$1$};
\end{tikzpicture} \ \ \ \ \ \ \ \
\begin{tikzpicture}
	\filldraw [black] (0,2) circle (2pt);
	\filldraw [black] (0,1) circle (2pt);
	\filldraw [black] (0.8,0) circle (2pt);
	\filldraw [black] (1.6,2) circle (2pt);
	\filldraw [black] (1.6,1) circle (2pt);
	\draw  [red,line width=0.8pt][dashed] (0,2)--(0,1);
	\node at (0,2.3) {$u_1$};
	\draw  [black](0,1)--(0.8,0)--(1.6,1)--(1.6,2)--(0,1);
	\draw  [black](0.8,0)--(1.6,2);
	\draw  [black](0.8,0)--(0,2)--(1.6,1);
	\node at (0.8, -1) {$\dot\Gamma_{2,2,1}$};
	\node at (0.8, -1.6) {$\lambda_1=1+\sqrt{3}$};
	\node at (-0.2, 2) {$1$};
	\node at (2.1, 2) {$\frac{\lambda_1+2}{\lambda_1}$};
	\node at (-0.2, 1) {$1$};
	\node at (2.1, 1) {$\frac{\lambda_1+2}{\lambda_1}$};
	\node at (1.3, 0) {$\frac{\lambda_1^2-2}{\lambda_1}$};
\end{tikzpicture} \ \ \ \ \ \ \ \
\begin{tikzpicture}
\filldraw [black] (0,2) circle (2pt);
\filldraw [black] (1,1) circle (2pt);
\filldraw [black] (1,0) circle (2pt);
\filldraw [black] (1,2) circle (2pt);
\filldraw [black] (2,2) circle (2pt);
\draw  [red,line width=0.8pt][dashed] (0,2)--(1,1);
\node at (0,2.3) {$u_1$};
\draw  [black](1,1)--(1,0)--(1,2)--(1,1)--(2,2);
\draw  [black](0,2)--(1,0)--(2,2);
\draw [black](1,2) .. controls (1.3,1.2) and (1.3,0.8).. (1,0);
\node at (1, -0.8) {$\dot\Gamma_{3,1,1}$};
\node at (1, -1.4) {$\lambda_1=\frac{1+\sqrt{17}}{2}$};
\node at (-0.2, 2) {$0$};
\node at (0.8, 2) {$1$};
\node at (2.2, 2) {$1$};
\node at (0.7, 1) {$\frac{\lambda_1}{2}$};
\node at (1.3, 0) {$\frac{\lambda_1}{2}$};
\end{tikzpicture}
\caption{Signed graphs $\dot\Gamma_{n_1, \dots, n_t}$  of order $4\le n\le 5$ and $3\le t\le n$.}
\label{f1}
\end{figure}

Suppose next that $n\ge 5$ and $t= 2$, or  $n\ge 6$ and $t\ge 3$.

Let $\mathbf{y}$ be any unit eigenvector associated with $\lambda_1(\dot\Gamma_{n_1, \dots, n_t})$.  
From $A(\dot\Gamma_{n_1, \dots, n_t})\mathbf{y}=\lambda_1(\dot\Gamma_{n_1, \dots, n_t})\mathbf{y}$, 
the entries at each vertex of $V_1\setminus \{u_1\}$ are equal, which we denote $y_1$. Similarly, denote by $y_2$ the entry at each vertex of $V_2\setminus \{u_2\}$, and  $y_i$ the entry at each vertex of $V_i$ for $i=3,\dots, t$, so 
\[
\mathbf{y}=(y_{u_1}, \underbrace{y_1,\dots,y_1}_{n_1-1}, y_{u_2}, \underbrace{y_2,\dots,y_2}_{n_2-1}, \underbrace{y_3,\dots,y_3}_{n_3}, \dots, \underbrace{y_t,\dots,y_t}_{n_t})^\top.
\]
For convenience, we write $\lambda_1$ for $\lambda_1(\dot\Gamma_{n_1, \dots, n_t})$.
Note that
\begin{align}
\lambda_1 y_{u_1}=-y_{u_2}+(n_2-1)y_2+\sum_{i=3}^tn_iy_i, \label{eq1}\\
\lambda_1 y_1=y_{u_2}+(n_2-1)y_2+\sum_{i=3}^tn_iy_i, \label{eq2}\\
\lambda_1 y_{u_2}=-y_{u_1}+(n_1-1)y_1+\sum_{i=3}^tn_iy_i, \label{eq3}\\
\lambda_1 y_2=y_{u_1}+(n_1-1)y_1+\sum_{i=3}^tn_iy_i. \label{eq4}
\end{align}
From $\eqref{eq1}$ and $\eqref{eq2}$, we have
\begin{equation}\label{v0}
y_{u_2}=\frac{\lambda_1}{2}(y_1-y_{u_1}).
\end{equation} 

We consider the cases $t=2$ and $t\ge 3$ seperately. 

Suppose  that $t=2$. 

Obviously, $n_2\ge 2$, and $n_1\ge 3$.  By \eqref{v0} and \eqref{eq3},  $\lambda_1^2(y_1-y_{u_1})=-2y_{u_1}+2(n_1-1)y_1$, i.e., 
\begin{equation}\label{add1}
y_{u_1}=\frac{\lambda_1^2-2(n_1-1)}{\lambda_1^2-2}y_1. 
\end{equation}
By Lemma \ref{gnt}, $\lambda_1$ is the largest root of $\phi_{n_1,n_2}(x)=0$ with $\phi_{n_1.n_2}(x)=x^4-n_1n_2x^2+4(n_1-1)(n_2-1)$, and so
\[
\lambda_1^2=\frac{n_1n_2}{2}+\frac{1}{2}\sqrt{n_1^2n_2^2-16(n_1-1)(n_2-1)}\ge 2(n_1-1)
\]
with equality if and only if $n_2=2$. This also implies that $\lambda_1^2>2$.  
If $n_2=2$, then from \eqref{add1}, 
 $y_{u_1}=0$,  from \eqref{eq1}, $y_{u_2}=y_2$, and from \eqref{v0},
 $y_1=\frac{2}{\lambda_1}y_{u_2}$, so 
 \[
\mathbf{y}=y_{u_2}\bigg(0,\underbrace{\frac{2}{\lambda_1},\dots,\frac{2}{\lambda_1}}_{n_1-1},\underbrace{1,\dots,1}_{n_2}\bigg)^\top,\]
which, by Chosing a positive $y_{u_2}$,   is non-negative with a unique $0$ entry at $u_1$. 
Suppose that $n_2\ge 3$. 
If $y_1=0$, then by \eqref{add1}, we have $y_{u_1}=0$, and so $y_{u_2}=y_2=0$ by \eqref{eq3} and \eqref{eq4}, respectively. Thus $\mathbf{y}=\mathbf{0}$, a contradiction. 
So $y_1\ne 0$. 
From \eqref{add1}, we have $y_1y_{u_1}>0$ and $|y_1|>|y_{u_1}|$. Assume that $y_1>0$. From \eqref{v0}, $y_{u_2}>0$. 
Similarly as above, we have $y_{u_1}=\frac{\lambda_1}{2}(y_2-y_{u_2})$ from 
\eqref{eq3} and \eqref{eq4},
and so 
\[
y_{u_2}=\frac{\lambda_1^2-2(n_2-1)}{\lambda_1^2-2}y_2
\]
from  \eqref{eq1}. Then $y_2>y_{u_2}>0$. Now from \eqref{eq1},  $y_{u_1}>0$, and $y_1>0$. 
So $\mathbf{y}$ is positive.  Then the result follows for $t=2$.

Suppose in the following  that $t\ge 3$. Then $n\ge 6$ and  $\dot\Gamma_{n_1, \dots, n_t}$ contains some signed induced subgraph of order $5$ in Fig. 1.  From Fig. 1, each such signed graph has index larger than $2$ except for $\lambda_1(\dot\Gamma_{1,1,3})=2$.
Then by Lemma \ref{interlace} and $\lambda_1(\dot\Gamma_{1,1,4})=2.3723> 2$, we have $\lambda_1(\dot\Gamma_{n_1, \dots, n_t})> 2$.

\begin{Claim} \label{Y1}
$y_1y_2\ge 0$ 
with equality if and only if $y_1=y_2=0$. 
\end{Claim}

\begin{proof}
From $y_{u_1}=\frac{\lambda_1}{2}(y_2-y_{u_2})$ and \eqref{v0}, we have
\begin{equation}\label{u0}
y_{u_1}=\frac{\lambda_1}{2}\left(y_2-y_{u_2}\right)=\frac{2\lambda_1y_2-\lambda_1^2y_1}{4-\lambda_1^2}.
\end{equation}
By adding $\eqref{eq1}$ to $\eqref{eq2}$ and adding $\eqref{eq3}$ to $\eqref{eq4}$, we have
\[
\frac{\lambda_1}{2}(y_1+y_{u_1})-(n_2-1)y_2=\sum_{i=3}^tn_iy_i=\frac{\lambda_1}{2}(y_2+y_{u_2})-(n_1-1)y_1, \]
i.e.,
\begin{equation}\label{eq7}
\lambda_1y_1+\lambda_1y_{u_1}+2(n_1-1)y_1-\lambda_1y_{u_2}=(\lambda_1+2n_2-2)y_2.
\end{equation}
Substituting $\eqref{v0}$ and $\eqref{u0}$ to $\eqref{eq7}$, we have
\[
\left((\lambda_1-2)(\lambda_1+2n_1-2)+\lambda_1^2\right)y_1=\left(2(\lambda_1-2)(\lambda_1+n_2-1)+2\lambda_1\right)y_2,
\]
i.e., 
\begin{equation}\label{y2}
y_2 
 =\frac{\lambda_1(\lambda_1-1)+(n_1-1)(\lambda_1-2)}{\lambda_1(\lambda_1-1)+(n_2-1)(\lambda_1-2)}y_1.
\end{equation}
As $\lambda_1>2$, we have  $\lambda_1(\lambda_1-1)+(n_1-1)(\lambda_1-2)>0$ and $(\lambda_1-2)(\lambda_1+n_2-1)+\lambda_1>0$. 
So the claim follows.
\end{proof}

\begin{Claim} \label{Y2}
 $y_{u_1}=0$ if $t=3$ and $n_2=n_3=1$, and 
$y_{u_1}y_1\ge 0$
with equality if and only if $y_{u_1}=y_1=0$ otherwise.
\end{Claim}

\begin{proof}
Let $f(x)=x^2+(n_2-2)x-2n_1+2$. 
From \eqref{u0} and \eqref{y2}, we have
\[
y_{u_1}=\frac{\lambda_1f(\lambda_1)}{(\lambda_1+2)(\lambda_1(\lambda_1-1)+(n_2-1)(\lambda_1-2))}y_1
\]
As $\lambda_1>2$,  $(\lambda_1+2)((\lambda_1-2)(\lambda_1+n_2-1)+\lambda_1)>0$. To prove the claim, it suffices to show that 
 $f(\lambda_1)\ge 0$ with equality if and only if $t=3$ and $n_2=n_3=1$.

If $n_1=1$, then $n_2=1$, and $f(\lambda_1)=\lambda_1(\lambda_1-1)>0$. 

Suppose that $n_1\ge 2$. 

\noindent{\bf Case 1.} $n_2=1$. 

Note that the largest root of $f(x)=0$ is $\alpha:=\frac{1}{2}\left(\sqrt{8n_1-7}+1\right)$. 
 
If $t=3$ and $n_3=1$, then by Lemma \ref{gnt} (ii), $\lambda_1=\tau_{n_1,1}=\alpha$, so $f(\lambda_1)=0$. 

Suppose that $t\ge 4$, or $t=3$ and $n_3\ge 2$. Let $\Gamma_1=\dot \Gamma_{n_1,\dots, n_t}[V_1\cup V_2\cup \{v_3, v_4\}]$ with $v_i\in V_i$ for $i=3,4$  if $t\ge 4$, and  $\Gamma_2=\dot \Gamma_{n_1,\dots, n_t}[V_1\cup V_2\cup \{v_3, v_3'\}]$ with $v_3, v_3'\in V_3$ if $t=3$ and $n_3\ge 2$. 
Then 
\[
\begin{pmatrix}
0 & 0 & -1 & 2 \\
0 & 0 & 1 & 2 \\
-1 & n_1-1 & 0 & 2 \\
1 & n_1-1 & 1 & 1 
\end{pmatrix}\ \ \ \ 
\text{and} \ \ \ \ 
\begin{pmatrix}
	0 & 0 & -1 & 2 \\
	0 & 0 & 1 & 2 \\
	-1 & n_1-1 & 0 & 2 \\
	1 & n_1-1 & 1 & 0 
\end{pmatrix},
\]
are the quotient matrices of $\Gamma_1$ and $\Gamma_2$, respectively.  Correspondingly, their characteristic polynomials are 
\[
f_1(x)=x^4-x^3-(3n_1+2)x^2+(8-3n_1)x+8n_1-8
\] 
and 
\[
f_2(x)=x^4-(3n_1+2)x^2-4(n_1-2)x+8n_1-8,
\]
respectively. 
Note that 
\[
f_1(\alpha)=-2(n_1-1)\left(n_1+\sqrt{8n_1-7}+1\right)<0
\]
and
\[
f_2(\alpha)=-\frac{1}{2}(n_1-1)\left(4n_1+3\sqrt{8n_1-7}-1\right)<0.
\]
By Lemma \ref{QM}, $\lambda_1(\Gamma_i)>\alpha$ for $i=1, 2$. 
By Lemma \ref{interlace},  $\lambda_1\ge \lambda_1(\Gamma_1)$ if $t\ge 4$, and 
 $\lambda_1\ge \lambda_1(\Gamma_2)$ if $t=3$ and $n_3\ge 2$. In either case,
 $\lambda_1>\alpha$,   so $f(\lambda_1)>0$.

\noindent{\bf Case 2.} $n_2=2$.

Then $f(x)=x^2-2n_1+2$.  By Lemma 2.7 (ii), $\lambda_1$ is larger than or equal to the largest root of 
$h_{n_1,2}=0$.
Note that for $n_1\ge 3$, 
\[
h_{n_1,2}\left(\sqrt{2n_1-2}\right)=-2\left((n_1-1)(n_1-2)+(n_1-2)\sqrt{2n_1-2}\right)<0. 
\]
So  $\lambda_1>\sqrt{2n_1-2}$, implying that $f(\lambda_1)>0$. 

Suppose that $n_1=2$. Then $f(x)=x^2-2$. As $n\ge 6$, $t\ge 4$ or $n_3\ge 2$ and $t=3$. 
Let $\Gamma_1=\dot\Gamma_{n_1, \dots,  n_t}[V_1\cup V_2\cup \{v_3, v_4\}]$ with $v_i\in V_i$ for $i=3, 4$ if $t\ge 4$, and $\Gamma_2=\dot\Gamma_{n_1, \dots,  n_t}[V_1\cup V_2\cup \{v_3, v_3'\}]$ with $v_3, v_3'\in V_3$ if $n_3\ge 2$ and $t=3$. Then 
\[
\begin{pmatrix}
	0 & 0 & -1 & 1 & 2 \\
	0 & 0 & 1 & 1 & 2 \\
	-1 & 1 & 0 &  0 & 2 \\
	1 & 1 & 0 &  0 & 2 \\
	1 & 1 & 1 & 1 & 1 
\end{pmatrix}\ \ \ \ 
\text{and} \ \ \ \ 
\begin{pmatrix}
	0 & 0 & -1 & 1 & 2 \\
	0 & 0 & 1 & 1 & 2 \\
	-1 & 1 & 0 & 0 & 2 \\
	1 & 1 & 0 & 0 & 2 \\
	1 & 1 & 1 & 1 & 0 
\end{pmatrix}
\]
are the quotient matrices of $\Gamma_1$ and $\Gamma_2$, respectively. Correspondingly, their characteristic polynomials are 
\[
f_1(x)=(x^2-2)(x^3-x^2-10x-6)
\]
and 
\[
f_2(x)=(x^2-2)(x^3-10x-8). 
\]
By a direct calculation, the largest root of $f_i(x)=0$ is larger than $\sqrt{2}$ for $i=1,2$. 
By Lemma \ref{QM}, $\lambda_1(\Gamma_i)>\sqrt{2}$ for $i=1,2$. 
By Lemma \ref{interlace}, $\lambda_1\ge \lambda_1(\Gamma_1)$ if $t\ge 4$, and $\lambda_1\ge \lambda_1(\Gamma_2)$ if $t=3$ and $n_3\ge 2$. 
In either case, $\lambda_1>\sqrt{2}$, so $f(\lambda_1)>0$.

\noindent{\bf Case 3.} $n_2\ge 3$. 

For $x\ge n_2$, let $g(x)=h_{n_1,n_2}(x)-x^2f(x)$, i.e., 
\[
g(x)= -n_2x^3+(n_1-n_2-n_1n_2+2)x^2+2(n_1+n_2-2)x-4(n_1+n_2-n_1n_2-1).
\]
As 
\begin{align*}
g'(x)& =-3n_2x^2+2(n_1-n_2-n_1n_2+2)x+2(n_1+n_2-2)\\
& \le g'(n_2)=-3n_2^3-2(n_1+1)n_2^2+2(n_1+3)n_2+2n_1-4\\
& <0, 
\end{align*}
we have
\[
g(x)\le g(n_2)=-((n_2^3-n_2^2-6n_2+4)n_1+n_2^4+n_2^3-4n_2^2+8n_2-4)<0. 
\]
So $x^2f(x)> h_{n_1, n_2}(x)$. 
By Lemma \ref{gnt} (ii), we have $f(\lambda_1)>0$.
\end{proof}

\begin{Claim}\label{Y3}
$y_{u_2}y_1\ge 0$ with equality if and only if $y_1=y_{u_2}=0$. 
\end{Claim}

\begin{proof}
Eliminating $y_{u_1}$ from \eqref{v0} and \eqref{u0}, we have
\[
\left(1-\frac{\lambda_1^2}{4}\right)y_{u_2}=\frac{\lambda_1}{2}y_1-\frac{\lambda_1^2}{4}y_2.
\]
So, from \eqref{y2}, we have
\begin{align*}
\left(4-\lambda_1^2\right)y_{u_2}& =\left(2\lambda_1-\frac{\lambda_1^2\left(\lambda_1(\lambda_1-1)+(n_1-1)(\lambda_1-2)\right)}{\lambda_1(\lambda_1-1)+(n_2-1)(\lambda_1-2)}\right)y_1\\
& =-\frac{\lambda_1(\lambda_1-2)(\lambda_1^2+(n_1-2)\lambda_1-2n_2+2)}{\lambda_1(\lambda_1-1)+(n_2-1)(\lambda_1-2)}y_1.
\end{align*}
Evidently, $\lambda_1(\lambda_1-2)>0$ and $(\lambda_1-2)(\lambda_1+n_2-1)+\lambda_1>0$. To prove the claim, it suffices to show that $\lambda_1^2+(n_1-2)\lambda_1-2n_2+2>0$.

 If $n_1=1$, then $n_2=1$, so $\lambda_1^2+(n_1-2)\lambda_1-2n_2+2=\lambda_1^2-\lambda_1>0$. 
Suppose that $n_1\ge 2$. By Lemma \ref{gnt} (ii), $\lambda_1\ge n_2$, so
\[
\lambda_1^2+(n_1-2)\lambda_1-2n_2+2\ge n_2^2+(n_1-2)n_2-2n_2+2\ge 2n_2^2-4n_2+2=2(n_2-1)^2\ge 0  
\]
with equalities if and only if $\lambda_1=n_2=n_1$ and $n_2=1$, contradicting the fact that $\lambda_1>2$. So 
 $\lambda_1^2+(n_1-2)\lambda_1-2n_2+2>0$ in any case.
\end{proof}

By Claims \ref{Y1}, \ref{Y2} and \ref{Y3}, $y_{u_1}=0$, $y_1, y_{u_2}, y_2$ have the same sign if $t=3$ and $n_2=n_3=1$ and $y_{u_1}, y_1, y_{u_2}, y_2$ have the same sign otherwise.

Let $|\mathbf{y}|$ the vector obtained from $\mathbf{y}$ by replacing any negative entry by its opposite number. 
By Rayleigh's principle, 
\begin{align*}
0 & \ge |\mathbf{y}|^\top A(\dot\Gamma_{n_1, \dots, n_t})|\mathbf{y}|-\mathbf{y}^\top A(\dot\Gamma_{n_1, \dots, n_t})\mathbf{y} = 2\sum_{v_iv_j\in E(\dot\Gamma_{n_1, \dots, n_t})}\sigma(v_iv_j)\left(|y_i||y_j|-y_iy_j\right)\\
& = 2\left(\sum_{v_iv_j\in E(\dot\Gamma_{n_1, \dots,  n_t})\setminus\{u_1u_2\}}\left(|y_iy_j|-y_iy_j\right)-\left(|y_{u_1}y_{u_2}|-y_{u_1}y_{u_2}\right)\right)\\
& =2\sum_{v_iv_j\in E(\Gamma_{n_1, \dots, n_t})\setminus\{u_1u_2\}}\left(|y_iy_j|-y_iy_j\right)\\
& \ge 0.
\end{align*}
So the above inequalities are equalities, implying that $|\mathbf{y}|$ is also an eigenvector associated with $\lambda_1(\dot\Gamma_{n_1, \dots, n_t})$. 
Thus there exists a non-negative eigenvector associated with $\lambda_1(\dot\Gamma_{n_1, \dots, n_t})$. This proves the first part.

By the above proof, we also have: 

(i)  if $t=2$, $n_2=2$ and $\mathbf{y}$ is non-negative, then   $\mathbf{y}$ has a unique zero entry at $u_1$;

(ii)  if $t=2$, $n_2\ge 3$ and $\mathbf{y}$ is non-negative, then   $\mathbf{y}$ is positive;

(iii) if $t=3$ and $n_2=n_3=1$, then $y_{u_1}=0$, $y_1, y_{u_2}, y_2$ have the same sign; otherwise, $y_{u_1}, y_1, y_{u_2}, y_2$ have the same sign.

To prove the second part, from (i) and (ii), we can assume that $t\ge 3$ and  $\mathbf{y}$ is a non-negative eigenvector associated with $\lambda_1(\dot\Gamma_{n_1, \dots, n_t})$.

We claim that  $y_1>0$. Otherwise,  $y_1=0$. 
Then $y_{u_1}=y_2=y_{u_2}=0$. 
From \eqref{eq1}, we have $\sum_{i=3}^tn_iy_i=0$, implying that $y_i=0$ for any $3\le i\le t$. 
So $\mathbf{y}=\mathbf{0}$, a contradiction.

For  $3\le i\le t$, as $\lambda_1y_i>y_1$, we have $y_i>0$.

If $t=3$ and $n_2=n_3=1$, then from (iii), $y_{u_1}=0$, and as $y_1>0$, we have $y_2, y_{u_2}>0$,  
so $\mathbf{y}$ is non-negative with the unique zero entry at $u_1$. 
Otherwise, from (iii), $y_{u_1}, y_2, y_{u_2}>0$, so $\mathbf{y}$ is positive. 
\end{proof}

Let $\boldsymbol\Gamma(n_1,\dots, n_t)$ be the class of unbalanced signed $t$-partite graphs  with partite sizes $n_1, \dots, n_t$, where $n_1\ge \dots\ge n_t$, and $t\ge 2$.

\section{Proof of Theorems \ref{N1} and \ref{N2}}

\begin{proof}[Proof of Theorem~\ref{N1}]
Let $\Gamma=(G,\sigma)$ be an unbalanced signed graph in $\boldsymbol\Gamma(n_1,\dots, n_t)$ that maximizes the index. 
Let $|V_i|=n_i$ for $i=1, \dots, t$, and $n_1+\dots+n_t=n$. 

As $\Gamma$ is unbalanced, we have $\Gamma\simeq \Gamma_{1,1,1}$ if $n=3$ and $t=3$, and $\Gamma\simeq \Gamma_{2,2}$ if $n=4$ and $t=2$. 
Suppose in the following that $n\ge 4$ and $t\ge 2$ with $(n,t)\ne (4,2)$. 

By Lemma \ref{L+},  $\Gamma$ is switching equivalent to
a signed graph $\Gamma'=(G,\sigma')$  such that $\lambda_1(\Gamma')$ has a non-negative unit eigenvector $\mathbf{x}$. 
By Lemma \ref{SE}, $\Gamma'\in \boldsymbol\Gamma(n_1,\dots, n_t)$ and $\lambda_1(\Gamma')=\lambda_1(\Gamma)$. 

\begin{Claim}\label{onene}
$\Gamma'$ contains exactly one negative edge. 
\end{Claim}

\begin{proof}
Suppose by contradiction  that $\Gamma'$ contains at least two negative edges. 
Construct $\Gamma^*$ from $\Gamma'$ by fixing one negative edge, say $u_0v_0$, and reversing the signs of all other negative edges in $\Gamma'$, and adding all possible edges between different partite sets as positive. So $\Gamma^*\in \boldsymbol\Gamma (n_1,\dots, n_t)$. Let $E_1$ be the set of negative edges of $\Gamma'$ except $u_0v_0$. 
Note that $\Gamma^*$ has a unique negative edge and $uv$ is positive in $\Gamma^*$ for any $uv\in E_1$. 
By Rayleigh's principle, we have
\[
0\ge \lambda_1(\Gamma^*)-\lambda_1(\Gamma')\ge \mathbf{x}^\top\left(A(\Gamma^*)-A(\Gamma')\right)\mathbf{x}\ge 4\sum_{uv\in E_1}x_{u}x_{v}\ge 0, 
\]
so $\lambda_1(\Gamma^*)=\lambda_1(\Gamma')$, implying that $\mathbf{x}$ is also an eigenvector associated with $\lambda_1(\Gamma^*)$ with  $x_{u}=0$ or $x_{v}=0$ for any $uv\in E_1$. Anyway, $\mathbf{x}$ is not positive.
Assume that $u_0\in V_i$ and $v_0\in V_j$ with $1\le i<j\le t$. 
Applying Lemma \ref{nne} to $\Gamma^*$, we find that  $t=2$ and $n_2=2$, 
or $t=3$, $n_1\ge 2$ and $n_2=n_3=1$, and in either case, $u_0$ is the unique vertex at which the entry of $\mathbf{x}$ is $0$. Then $E_1$ consists of a single edge $u_0w$ for some $w$. But  switching $\{u_0\}$ in $\Gamma'$ creates  $K_{n_1,2}^+$ if $t=2$ and $K_{n_1,1,1}^+$ if $t=3$, which is balanced in either case, a contradiction. 
\end{proof}

 By Claim \ref{onene}, $\Gamma'$ contains exactly one negative edge, say  $u_0v_0$ with $u_0\in V_i$ and $v_0\in V_j$, where $i<j$.

As $G$ is a $k$-partite graph with partite sizes $n_1,\dots, n_t$, $G$ is a subgraph of $K_{n_1,\dots, n_t}$. 

\begin{Claim}\label{complete}
$G= K_{n_1, \dots, n_t}$.
\end{Claim}

\begin{proof}
Suppose to the contrary that $G$ is a proper spanning subgraph of $K_{n_1, \dots, n_t}$. 
Let  $E_1:=E(K_{n_1, \dots, n_t})\setminus E(G)$. 
Let $\Gamma''=(G_1,\sigma'')$ be a signed graph obtained from $\Gamma'$ by adding all  edges in $E_1$ as positive edges, so $G_1=K_{n_1, \dots, n_t}$ and  $\Gamma''\in \boldsymbol\Gamma (n_1,\dots, n_t)$. 
By Rayleigh's principle, 
\[
0\ge \lambda_1(\Gamma'')-\lambda_1(\Gamma')\ge \mathbf{x}^\top \left(A(\Gamma'')-A(\Gamma')\right)\mathbf{x}=2\sum_{uv\in E_1}x_ux_v\ge0, 
\]
so $\lambda_1(\Gamma'')=\lambda_1(\Gamma')$, and $\mathbf{x}$ is also an eigenvector associated with $\lambda_1(\Gamma'')$. 
Let $u'v'\in E_1$. 
Then  $\{u', v'\}\ne \{u_0, v_0\}$. Assume that $u'\in \{u', v'\}\setminus \{u_0, v_0\}$. 
As 
\[
0=\left(\lambda_1(\Gamma'')-\lambda_1(\Gamma')\right)x_{v'}=\sum_{uv'\in E_1}x_u\ge x_{u'}, 
\]
so $x_{u'}=0$, contradicting  Lemma \ref{nne} as $u'\ne u_0$. 
\end{proof}

By Claim \ref{complete}, $G=K_{n_1, \dots, n_t}$. 

If $t=2$, then $\Gamma\simeq \Gamma_{n_1, n_2}$, as desired. And if $t\ge 3$ and $n=3,4,5$, then from Fig.~\ref{f1}, $\Gamma'\simeq \Gamma_{n_1, \dots, n_t}$, as desired. 

Suppose in the following that $t\ge 3$ and $n\ge 6$. 
By Lemmas \ref{interlace} and \ref{gnt}, $\lambda_1(\Gamma')\ge\lambda_1(\dot\Gamma_{n_1, \dots, n_t})\ge n_j$. Recall that $n_1\ge \dots\ge n_t$ and $\Gamma'$ contains exactly one negative edge connecting  $u_0\in V_i$ and $v_0\in V_j$.

\begin{Claim}\label{large}
$n_i=n_1$ and $n_j=n_2$.  
\end{Claim}

\begin{proof}
Suppose to the contrary that $n_i\ne n_1$ or $n_j\ne n_2$.
Let 
\[
k=\begin{cases}
1 &  \mbox{if $n_i\ne n_1$, or $i\ge 2$ and $n_i=n_1=n_2$},\\
2 &  \mbox{otherwise}.
\end{cases}
\]
Then $n_k> n_i$ or $n_k> n_j$, so $n_k\ge 2$. By the choice of $k$, $k\ne i,j$. 
By Lemma \ref{nne}, $\mathbf{x}$ is positive. 
For any $u, v\in V_i\setminus \{u_0\}$, we have  $\lambda_1(\Gamma')x_u=\lambda_1(\Gamma')x_v$, so $x_u=x_v$, which we denote by $x_i$. 
Similarly, the entry of $\mathbf{x}$ at each vertex in $V_j\setminus\{v_0\}$ is the same and is denoted by $x_j$,  and  the entry of $\mathbf{x}$ at each vertex in $V_\ell$ is the same and is denoted by $x_\ell$
for any $1\le \ell \le t$ and $\ell\ne i, j$. 

Let $u_1\in V_k$.  
Let $\dot{\Gamma}$ be a signed graph obtained from $\Gamma'$ by reversing the sign of $u_0v_0$ and $u_0u_1$. Then $\dot{\Gamma}\in \boldsymbol\Gamma(n_1,\dots, n_t)$. By Lemma \ref{nne}, let $\mathbf{y}$ be a non-negative unit eigenvector associated with $\lambda_1(\dot{\Gamma})$. Denote by $y_i$ the entry of $\mathbf{y}$ at each vertex in $V_i\setminus\{u_0\}$, $y_k$ the entry of $\mathbf{y}$ at each vertex in $V_k\setminus\{u_1\}$ and $y_\ell$ the entry of $\mathbf{y}$ at each vertex in $V_\ell$ for any $1\le \ell \le t$ and $\ell\ne i, k$. 

If $t=3$ and $n_i=1$, then $n_j=1$, so by Lemma \ref{nne}, $y_{u_1}=0$. Since $\lambda_1(\dot{\Gamma})(y_j-y_{u_0})=y_{u_0}+2y_{u_1}-y_j=y_{u_0}-y_j$, we have $(\lambda_1(\dot{\Gamma})+1)(y_j-y_{u_0})=0$, i.e., $y_j=y_{u_0}$. 
As $\lambda_1(\Gamma')(x_{u_0}-x_{v_0})=-x_{v_0}+x_{u_0}$, we have $x_{u_0}=x_{v_0}$, and $\lambda_1(\Gamma')x_k=x_{u_0}+x_{v_0}=2x_{u_0}$. 
So $x_k=\frac{2}{\lambda_1(\Gamma')}x_{u_0}<x_{u_0}$. 
By Rayleigh's principle, 
\begin{align*}
0& \ge \mathbf{y}^\top\mathbf{x}(\lambda_1(\dot{\Gamma})-\lambda_1(\Gamma'))= \mathbf{y}^\top\left(A(\dot{\Gamma})-A(\Gamma')\right)\mathbf{x}=2\left(y_{u_0}x_{v_0}+x_{u_0}y_j-y_{u_1}x_{u_0}-x_ky_{u_0}\right)\\
& =2\left(y_{u_0}x_{v_0}+x_{u_0}y_{u_0}-x_ky_{u_0}\right)>0, 
\end{align*}
a contradiction.

Suppose that $t\ge 4$ or $n_i\ge 2$. Then by Lemma \ref{nne}, $\mathbf{y}$ is positive. 
For convenience, denote by $\lambda_1=\lambda_1(\Gamma)=\lambda_1(\Gamma')$ and $\lambda_1'=\lambda_1(\dot{\Gamma})$.  
Note that if $n_i\ge 2$, then $\lambda_1(x_i-x_{u_0})=2x_{v_0}$, so  
\[
x_i=x_{u_0}+\frac{2}{\lambda_1}x_{v_0}
\]
and  similarly, if $n_j\ge 2$, then  
\[
x_j=x_{v_0}+\frac{2}{\lambda_1}x_{u_0}. 
\]
As
\[
\lambda_1x_{v_0}=-x_{u_0}+(n_i-1)x_i+\sum_{1\le \ell\le t\atop\ell\ne i, j}n_\ell x_\ell
\]
and
\[
\lambda_1x_{u_0} = -x_{v_0}+(n_j-1)x_j+\sum_{1\le \ell\le t\atop\ell\ne i, j}n_\ell x_\ell,
\]
we have
\begin{align*}
\lambda_1x_{u_0}
& = (\lambda_1-1)x_{v_0}+(n_j-1)\left(x_{v_0}+\frac{2}{\lambda_1}x_{u_0}\right)+x_{u_0}-(n_i-1)\left(x_{u_0}+\frac{2}{\lambda_1}x_{v_0}\right)\\
& = \left(\lambda_1+n_j-2-\frac{2(n_i-1)}{\lambda_1}\right)x_{v_0}-\left(n_i-2-\frac{2(n_j-1)}{\lambda_1}\right)x_{u_0}, 
\end{align*}
i.e.,
\[
\left(\lambda_1+n_i-2-\frac{2(n_j-1)}{\lambda_1}\right)x_{u_0}=\left(\lambda_1+n_j-2-\frac{2(n_i-1)}{\lambda_1}\right)x_{v_0}, 
\]
As $\lambda_1\ge n_j$, we have $\lambda_1^2+(n_i-2)\lambda_1-2(n_j-1)>0$, so
\[
x_{v_0}=\frac{\lambda_1^2+(n_i-2)\lambda_1-2(n_j-1)}{\lambda_1^2+(n_j-2)\lambda_1-2(n_i-1)}x_{u_0}. 
\]
Similarly, as
\[
\lambda_1x_k=x_{v_0}+(n_j-1)x_j+x_{u_0}+(n_i-1)x_i+\sum_{1\le \ell\le t\atop\ell\ne i, j}n_\ell x_\ell-n_kx_k,
\]
we also have
\begin{align*}
\lambda_1x_k
& =x_{v_0}+(n_j-1)\left(x_{v_0}+\frac{2}{\lambda_1}x_{u_0}\right)+x_{u_0}+\lambda_1x_{v_0}+x_{u_0}-n_kx_k\\
& =(\lambda_1+n_j)x_{v_0}+\left(2+\frac{2(n_j-1)}{\lambda_1}\right)x_{u_0}-n_kx_k\\
& =(\lambda_1+n_j)\left(\frac{\lambda_1^2+(n_i-2)\lambda_1-2(n_j-1)}{\lambda_1^2+(n_j-2)\lambda_1-2(n_i-1)}+2+\frac{2(n_j-1)}{\lambda_1}\right)x_{u_0}-n_kx_k, 
\end{align*}
so
\[
x_k=\left(\frac{\lambda_1+n_j}{\lambda_1+n_k}\cdot\frac{\lambda_1^2+(n_i-2)\lambda_1-2(n_j-1)}{\lambda_1^2+(n_j-2)\lambda_1-2(n_i-1)}+\frac{2}{\lambda_1+n_k}+\frac{2(n_j-1)}{\lambda_1(\lambda_1+n_k)}\right)x_{u_0}. 
\]
Thus 
\[
x_{v_0}-x_k=\left(\left(1-\frac{\lambda_1+n_j}{\lambda_1+n_k}\right)\cdot\frac{\lambda_1^2+(n_i-2)\lambda_1-2(n_j-1)}{\lambda_1^2+(n_j-2)\lambda_1-2(n_i-1)}-\frac{2}{\lambda_1+n_k}-\frac{2(n_j-1)}{\lambda_1(\lambda_1+n_k)}\right)x_{u_0}. 
\]
By Lemma \ref{gnt}, $\lambda_1'\ge \min\{n_i, n_k\}$. 
Similarly, we have 
\[
y_{u_1}=\frac{\lambda_1'^2+(n_i-2)\lambda_1'-2(n_k-1)}{\lambda_1'^2+(n_k-2)\lambda_1'-2(n_i-1)}y_{u_0},  
\]
and 
\[
y_j=\left(\frac{\lambda_1'+n_k}{\lambda_1'+n_j}\cdot\frac{\lambda_1'^2+(n_i-2)\lambda_1'-2(n_k-1)}{\lambda_1'^2+(n_k-2)\lambda_1'-2(n_i-1)}+\frac{2}{\lambda_1'+n_j}+\frac{2(n_k-1)}{\lambda_1'(\lambda_1'+n_j)}\right)y_{u_0}. 
\]
So
\[
y_j-y_{u_1}=\left(\left(\frac{\lambda_1'+n_k}{\lambda_1'+n_j}-1\right)\cdot\frac{\lambda_1'^2+(n_i-2)\lambda_1'-2(n_k-1)}{\lambda_1'^2+(n_k-2)\lambda_1'-2(n_i-1)}+\frac{2}{\lambda_1'+n_j}+\frac{2(n_k-1)}{\lambda_1'(\lambda_1'+n_j)}\right)y_{u_0}.
\]
Note that
\[
\frac{2}{\lambda_1'+n_j}-\frac{2}{\lambda_1+n_k}\ge \frac{2}{\lambda_1+n_j}-\frac{2}{\lambda_1+n_k}>0
\]
and 
\[
\frac{2(n_k-1)}{\lambda_1'(\lambda_1'+n_j)}-\frac{2(n_j-1)}{\lambda_1(\lambda_1+n_k)}\ge \frac{2(n_k-1)}{\lambda_1(\lambda_1+n_j)}-\frac{2(n_j-1)}{\lambda_1(\lambda_1+n_k)}>0.
\]
By Rayleigh's principle, 
\begin{align*}
0& \ge \mathbf{y}^\top\mathbf{x}(\lambda_1(\dot{\Gamma})-\lambda_1(\Gamma'))\\
& = \mathbf{y}^\top\left(A(\dot{\Gamma})-A(\Gamma')\right)\mathbf{x}=2\left(y_{u_0}x_{v_0}+x_{u_0}y_j-y_{u_1}x_{u_0}-x_ky_{u_0}\right)\\ 
& =2\left((x_{v_0}-x_k)y_{u_0}+(y_j-y_{u_1})x_{u_0}\right)\\
& =2x_{u_0}y_{u_0}\left(\frac{n_k-n_j}{\lambda_1+n_k}\cdot\frac{\lambda_1^2+(n_i-2)\lambda_1-2(n_j-1)}{\lambda_1^2+(n_j-2)\lambda_1-2(n_i-1)}-\frac{2}{\lambda_1+n_k}-\frac{2(n_j-1)}{\lambda_1(\lambda_1+n_k)}\right.\\
& \quad\left.+\frac{n_k-n_j}{\lambda_1'+n_j}\cdot\frac{\lambda_1'^2+(n_i-2)\lambda_1'-2(n_k-1)}{\lambda_1'^2+(n_k-2)\lambda_1'-2(n_i-1)}+\frac{2}{\lambda_1'+n_j}+\frac{2(n_k-1)}{\lambda_1'(\lambda_1'+n_j)}\right)\\
& >0, 
\end{align*}
a contradiction. 
\end{proof}

By Claims \ref{onene}, \ref{complete} and \ref{large}, we have $\Gamma'\simeq \Gamma_{n_1, \dots, n_t}$. 
So  $\lambda_1(\Gamma)=\lambda_1(\Gamma')=\lambda_1(\Gamma_{n_1, \dots, n_t})$, and $\Gamma\simeq \Gamma_{n_1, \dots, n_t}$.

By Lemma \ref{gnt}, 
$\lambda_1(\dot\Gamma_{n_1, \dots, n_t})$ is the largest root of $\phi_{n_1, \dots, n_t}(x)=0$. 
\end{proof}


\begin{proof}[Proof of Theorem~\ref{N2}]
Let $\Gamma=(G,\sigma)$ be an unbalanced signed $t$-partite graph of order $n$  maximizes the index. 
Let $V_1, \dots, V_t$ be the partite sets of $G$ with  $|V_i|=n_i$ for $i=1,\dots, t$. Assume that $n_1\ge \dots\ge n_t$. 
By Theorem $\ref{N1}$, $\Gamma\simeq \Gamma_{n_1, \dots, n_t}$. 
Let $u_0v_0$ be the unique negative edge of $\Gamma_{n_1, \dots, n_t}$. 
Then by Theorem $\ref{N1}$ again, we can assume that  $u_0\in V_1$ and $v_0\in V_2$. 

\begin{Claim}\label{turan}
$n_1-n_t=0,1$.
\end{Claim}

\begin{proof}
Suppose to the contrary that $n_1-n_t\ge 2$. 
Let $u'\in V_1\setminus\{u_0\}$. 
Let $\Gamma'$ be the signed graph obtained from $\Gamma$ by deleting the edges between $u'$ and each vertex in $V_t$, and adding all possible edges between $u'$ and each vertex in $V_1\setminus\{u'\}$ as positive ones. 
Then $\Gamma'\in \boldsymbol\Gamma (n_1-1, n_2, \dots, n_t+1)$. 
By Lemma \ref{nne}, there are non-negative eigenvectors  $\mathbf{x}$ and $\mathbf{y}$ associated with $\lambda_1(\Gamma)$ and $\lambda_1(\Gamma')$,  respectively. 
As $\lambda_1(\Gamma)x_u=\lambda_1(\Gamma)x_v$ for any $u, v\in V_\ell\setminus\{u_0, v_0\}$ with $\ell=1, \dots, t$, we have $x_u=x_v$. 
Denote by $x_\ell=x_u$ for any $u\in V_\ell\setminus\{u_0, v_0\}$ with $\ell=1, \dots, t$. 
Similarly, denote by $y_\ell=y_u$ for any $u\in V_\ell\setminus\{u_0, v_0\}$ with $\ell=1, \dots, t$. 

For convenience, denote by $\lambda_1=\lambda_1(\Gamma)$ and $\lambda_1'=\lambda_1(\Gamma')$. 
By Raleigh's principle, 
\begin{align*}
0& \ge\mathbf{y}^\top(\lambda_1'-\lambda_1)\mathbf{x} =\mathbf{y}^\top(A(\Gamma')-A(\Gamma))\mathbf{x}\\
& =x_{u_0}y_t+y_{u_0}x_1+(n_1-2)(x_1y_t+y_1x_1)-n_t(x_1y_t+y_tx_t)\\
& =(x_{u_0}-n_tx_t)y_t+(y_{u_0}+(n_1-2)y_1)x_1+(n_1-n_t-2)x_1y_t\\
& =(\lambda_1x_t-(\lambda_1+n_1-1)x_1)y_t+((\lambda_1'+n_t+1)y_t-\lambda_1'y_1)x_1+(n_1-n_t-2)x_1y_t\\
& =(\lambda_1'-\lambda_1)x_1y_t+\lambda_1x_ty_t-\lambda_1'x_1y_1
\end{align*}

Suppose first that $t=2$. 
Note that 
\[
\lambda_1(x_2-x_1)=x_{u_0}+(n_1-1)x_1-x_{v_0}-(n_2-1)x_2
\]
and
\[
(\lambda_1-1)(x_{u_0}-x_{v_0})=(n_2-1)x_2-(n_1-1)x_1.
\]
As $n_1-n_2\ge 2$, $\Gamma$ contains $\Gamma_{4,2}$ as an induced subgraph, and so $\lambda_1\ge \lambda_1(\Gamma_{4,2})=2.4495>2$. 
Since $(n_1-2)\lambda_1-2(n_1-1)-(n_2-2)\lambda_1+2(n_2-1)=(\lambda_1-2)(n_1-n_2)>0$, we have 
\[
\left(\lambda_1^2+(n_2-2)\lambda_1-2(n_2-1)\right)x_2=\left(\lambda_1^2+(n_1-2)\lambda_1-2(n_1-1)\right)x_1,  
\]
so $x_2>x_1$. 
Note that 
\[
\lambda_1'(y_2-y_1)=y_{u_0}+(n_1-2)y_1-y_{v_0}-n_2y_2
\]
and
\[
\lambda_1'(y_{u_0}-y_{v_0})=n_2y_2-(n_1-2)y_1. 
\]
As $n_1-n_2\ge 2$, $\Gamma'$ contains $\Gamma_{3,3}$ as an induced subgraph, and so $\lambda_1'\ge \lambda_1(\Gamma_{3,3})=2.5616>2$. 
Since $(n_1-3)\lambda_1-2(n_1-2)-(n_2-1)\lambda_1+2n_2=(\lambda_1'-2)(n_1-n_2-2)\ge 0$, we have 
\[
\left(\lambda_1'^2+(n_2-1)\lambda_1'-2n_2\right)y_2=\left(\lambda_1'^2+(n_1-3)\lambda_1'-2(n_1-2)\right)y_1,  
\]
so $y_2\ge y_1$. Thus
\[
0 \ge\mathbf{y}^\top(\lambda_1'-\lambda_1)\mathbf{x}>(\lambda_1'-\lambda_1+\lambda_1-\lambda_1')x_1y_t=0, 
\]
a contradiction. 

Suppose next that $t\ge 3$. 
Note that 
\[
\lambda_1(x_t-x_1)=x_{u_0}+(n_1-1)x_1-n_tx_t
\]
and 
\[
\lambda_1'(y_t-y_1)=y_{u_0}+(n_1-2)y_1-(n_t+1)y_t.
\]
Then 
\[
x_t=\frac{\lambda_1+n_1-1}{\lambda_1+n_t}x_1+\frac{x_{u_0}}{\lambda_1+n_t}\ge \frac{\lambda_1+n_1-1}{\lambda_1+n_t}x_1\ge \frac{\lambda_1+n_t+1}{\lambda_1+n_t}x_1
\]
and 
\[
y_t=\frac{\lambda_1'+n_1-2}{\lambda_1'+n_t+1}y_1+\frac{y_{u_0}}{\lambda_1'+n_t+1}\ge \frac{\lambda_1'+n_1-2}{\lambda_1'+n_t+1}y_1\ge \frac{\lambda_1'+n_t}{\lambda_1'+n_t+1}y_1. 
\]
Thus
\begin{align*}
0& \ge\mathbf{y}^\top(\lambda_1'-\lambda_1)\mathbf{x}  \ge\left(\lambda_1'-\lambda_1+\frac{\lambda_1(\lambda_1+n_t+1)}{\lambda_1+n_t}-\frac{\lambda_1'(\lambda_1'+n_t+1)}{\lambda_1'+n_t}\right)x_1y_t\\
& =\frac{(\lambda_1-\lambda_1')}{(\lambda_1+n_t)(\lambda_1'+n_t)}x_1y_t
 \ge 0, 
\end{align*}
so $\lambda_1'=\lambda_1$, implying that $n_1=n_t+2$ and $x_{u_0}=y_{u_0}=0$. As $x_{u_0}=0$, we have 
$t=3$ and $n_2=n_3=1$ by Lemma \ref{nne},  so $n_1=3$. But $y_{u_0}=0$, contradicting Lemma \ref{nne}. 
\end{proof}

By Claim \ref{turan}, we have $\Gamma\simeq \Gamma_n(t)$, and $\lambda_1(\Gamma)=\lambda_1(\Gamma_n(t))$. 
\end{proof}

\section{Proof of Theorems \ref{N3}, \ref{N4}, \ref{N5} and \ref{N6}}


\begin{proof} [Proof of Theorem \ref{N3}]
Let $\Gamma=(G, \sigma)$ be an unbalanced signed $t$-partite graph with partite sizes $n_1\ge \dots, \ge n_t$ that minimizes the least eigenvalue. 
Let $n=\sum_{i=1}^{t}n_i$. 
By Lemma \ref{L+},  $\Gamma$ is switching equivalent to
a signed graph $\Gamma'=(G,\sigma')$  such that $\lambda_n(\Gamma')$ has a non-negative unit eigenvector $\mathbf{x}$. Assume that $V(G)=V(K_{n_1, \dots, n_t})$.

Since $K_{n_1, \dots, n_t}^-\in \boldsymbol\Gamma(n_1,\dots, n_t)$, we have $\lambda_n(K_{n_1, \dots, n_t}^-)\ge\lambda_n(\Gamma)$ by the choice of $\Gamma$. 
Let $E_1=E(K_{n_1, \dots, n_t})\setminus E(G)$, and 
$E_2=\{e\in E(G): \mbox{$e$ is positive in $\Gamma'$}\}$. 
Suppose that $E_1\cup E_2\ne \emptyset$. 
By Rayleigh's principle,  
\[
0\le \lambda_n(K_{n_1, \dots, n_t}^-)-\lambda_n(\Gamma')\le \mathbf{x}^\top\left(A(K_{n_1, \dots, n_t}^-)-A(\Gamma')\right)\mathbf{x}=-2\sum_{uv\in E_1}x_ux_v-4\sum_{uv\in E_2}x_ux_v\le 0,  
\]
so $\lambda_n(K_{n_1, \dots, n_t}^-)=\lambda_n(\Gamma)$, implying that $\mathbf{x}$ is also an eigenvector associated with $\lambda_n(K_{n_1,\dots,n_t}^-)$ with $x_u=0$ or $x_v=0$ for any $uv\in E_1\cup E_2$. Assume that $x_u=0$ and $u\in V_i$ for some $uv\in E_1\cup E_2$, where $1\le i\le t$. 
Then for any $w\in V_i\setminus\{u\}$, $x_w=0$ as $\lambda_n(K_{n_1,\dots,n_t}^-)x_w=\lambda_n(K_{n_1,\dots,n_t}^-)x_u$. 
Note that
\[
0=\lambda_n(K_{n_1,\dots,n_t}^-)x_u=-\sum_{j=1,j\ne i}^t\sum_{w\in V_j}x_w. 
\]
Then $x_w=0$ for any $w\in V_j$ with $1\le j\le t$ and $j\ne i$, so $\mathbf{x}=\mathbf{0}$, a contradiction. 
Therefore $E_1\cup E_2=\emptyset$, implying that $\Gamma\simeq K_{n_1, \dots, n_t}^-$, and $\lambda_n(\Gamma)=\lambda_n\left(K_{n_1, \dots, n_t}^-\right)$.
%
%
\end{proof}

\begin{proof}[Proof of Theorem~\ref{N5}] 
Let $\Gamma$ be an unbalanced signed $t$-partite graph with partite sizes $n_1\ge \dots\ge n_t$ maximizes the spectral radius. 
Note that $\rho(\Gamma)=\max\{\lambda_1(\Gamma), -\lambda_n(\Gamma)\}$. 

If $t=2$, then  
so  $\rho(\Gamma)=\lambda_1(\Gamma)=-\lambda_n(\Gamma)$,  so the result follows from Theorem \ref{N1} by noting that $\rho(\Gamma_{n_1,n_2})$ is the largest root of $\phi_{n_1,\dots,n_t}(x)=0$ with $\phi_{n_1,\dots,n_t}(x)=x^4-n_1n_2x^2+4(n_1-1)(n_2-1)$.

Suppose  that $t\ge 3$. 
By Theorems \ref{N1} and \ref{N3}, we have 
\[
\rho(\Gamma)=\max\{\lambda_1(\Gamma_{n_1, \dots,  n_t}), -\lambda_n(K_{n_1, \dots, n_t}^-)\}.
\] 
By Lemma \ref{under},  
\[
\lambda_1(\Gamma_{n_1, \dots, n_t})<\lambda_1(K_{n_1, \dots, n_t}^+)=-\lambda_n(K_{n_1, \dots, n_t}^-). 
\]
So $\rho(\Gamma)=-\lambda_n(K_{n_1, \dots, n_t}^-)$. Now by Theorem \ref{N3}, $\Gamma\simeq K_{n_1, \dots, n_t}^-$. From \cite{De,EH}, 
$\rho(K_{n_1, \dots, n_t}^-)=\lambda_1(K_{n_1, \dots, n_t}^+)$ is equal to the largest root of $\left(\sum_{i=1}^t\frac{n_i}{x+n_i}-1\right)\prod_{i=1}^t(x+n_i)=0$.
\end{proof}



\begin{proof}[Proof of Theorem \ref{N4}]
Let $\Gamma$ be an unbalanced signed $t$-partite signed graph of order $n$ that minimizes the least eigenvalue with partite sizes $n_1\ge \dots\ge n_t$. 
By Theorem \ref{N3}, $\lambda_n(\Gamma)=-\lambda_1(K_{n_1, \dots, n_t})$, and $\Gamma\simeq K_{n_1, \dots, n_t}^-$. 
Note that $\lambda_1(K_{n_1, \dots, n_t}^+)\le \lambda_1(T_{n, t}^+)$ with equality if and only if $K_{n_1, \dots, n_t}$ is isomorphic to $T_{n, t}$ by \cite{SGR}.
So $\lambda_n(\Gamma)\ge -\lambda_1(T_{n, t}^+)=\lambda_n(T_{n, t}^-)$. On the other hand, we have by the choice of $\Gamma$,  $\lambda_n(\Gamma)\le \lambda_n(T_{n, t}^-)$.
Thus $\lambda_n(\Gamma)=\lambda_n(T_{n, t}^-)$ and $\Gamma\simeq T_{n, t}^-$.
\end{proof}

\begin{proof}[Proof of Theorem~\ref{N6}]
Let $\Gamma$ be an unbalanced signed $t$-partite graph of order $n$ that maximizes the spectral radius. Note that $\rho(\Gamma)=\max\{\lambda_1(\Gamma), -\lambda_n(\Gamma)\}$. 

If $t=2$, then  $\rho(\Gamma)=\lambda_1(\Gamma)=-\lambda_n(\Gamma)$, so the result follows from Theorem \ref{N2}. 

Suppose  that $t\ge 3$. 
By Theorems \ref{N2} and \ref{N4}, we have \[
\rho(\Gamma)=\max\{\lambda_1(\Gamma_n(t)), -\lambda_n(T_{n, t}^-)\}.
\]
By Lemma \ref{under},  
\[
\lambda_1(\Gamma_n(t))<\lambda_1(T_{n, t}^+)=-\lambda_n(T_{n, t}^-). 
\]
So $\rho(\Gamma)=-\lambda_n(T_{n, t}^-)$, and by Theorem \ref{N4}, $\Gamma\simeq T_{n, t}^-$. 
\end{proof}

\bigskip

\noindent {\bf Acknowledgements.}
The authors thank Professor Dragan Stevanovi\'{c} for his very  helpful comments and suggestions on an early version. 
This work was supported by the National Natural Science Foundation of China (No.~12571364).


\begin{thebibliography}{99}

\bibitem{AL}  F.M. Atay, S. Liu, 
Cheeger constants, structural balance, and spectral clustering analysis for signed graphs,
Discrete Math. 343 (1) (2020)  111616, 26 pp.

\bibitem{BB} F. Belardo, M. Brunetti, Limit points for the spectral radii of signed graphs, Discrete Math. 347 (2) (2024) 113745.

\bibitem{BH}
A.E. Brouwer, W.H. Haemers, Spectra of Graphs,  Springer, New York, 2012.


\bibitem{BS}
M. Brunetti, Z. Stani\'c, Ordering signed graphs with large index,
Ars Math. Contemp. 22 (4) (2022)  14 pp.

\bibitem{CDG}
C.M. Conde,  E. Dratman, L.N. Grippo,
On the spectral radius of unbalanced signed bipartite graphs,
Discrete Math. 349 (2026) 114942.

\bibitem{De} C. Delorme, 
Eigenvalues of complete multipartite graphs, Discrete Math. 312 (2012) 2532--2535.

\bibitem{EH}
F. Esser, F. Harary, 
On the spectrum of a complete multipartite graph, 
Europ. J. Combin. 1 (1980) 211--218. 

\bibitem{Ha} F. Harary, 
On the notion of balance of a signed graph,
Michigan Math. J. 2 (1953/54)  143--146.

\bibitem{HK} F. Harary, J.A. Kabell, 
A simple algorithm to detect balance in signed graphs, 
Math. Social Sci. 1 (1) (1980/81) 131--136.  

\bibitem{HJ}
R. Horn, C. Johnson, Matrix Analysis, Second ed., 
Cambridge Univ. Press, Cambridge, 2013. 

\bibitem{Ni} V. Nikiforov, 
Bounds on graph eigenvalues. II,
Linear Algebra Appl. 427 (2--3) (2007) 183--189.

\bibitem{SL}
G. Sun, F. Liu, K. Lan,
A note on eigenvalues of signed graphs,
Linear Algebra Appl. 652 (2022) 125--131.


\bibitem{StanB} Z. Stani\'c, Spectra of Signed Graphs, 
London Math. Soc. Lecture Note Ser., 504, 
Cambridge University Press, Cambridge, 2026.

\bibitem{Stan}
Z. Stani\'c, Perturbations in a signed graph and its index, Discuss. Math. Graph Theory 38
(2018) 841--852.


\bibitem{St1}
Z. Stani\'c, Integral regular net-balanced signed graphs with vertex degree at most four, Ars Math. Contemp. 17 (2019) 103--114.

\bibitem{SGR} D. Stevanovi\'c, I.  Gutman,  M.U. Rehman, 
On spectral radius and energy of complete multipartite graphs,
Ars Math. Contemp. 9 (1) (2015) 109--113.


\bibitem{Z1}
T. Zaslavsky,
Signed graphs,
Discrete Appl. Math. 4 (1982) 47--74.


\bibitem{Z2} T. Zaslavsky, Biased graphs, I: bias, balance, and gains, J. Combin. Theory Ser. B 47 (1989) 32--52.



\end{thebibliography}
\end{document}